\documentclass[letterpaper, 11pt]{amsart}

\usepackage{fullpage} 
\usepackage{graphicx} 
\graphicspath{ {./figs/} } 
\usepackage{amsmath, amssymb, amsfonts, amsthm}
\usepackage{mathtools, mathrsfs} 
\usepackage{enumerate}
\usepackage{empheq} 
\usepackage{caption} 

\usepackage{color}
\definecolor {darkblue} {rgb} {0,0,0.4} 
\definecolor {dg} {rgb} {0,0.392,0} 
\usepackage{url} 
\usepackage[plainpages = false, colorlinks, citecolor = blue, filecolor = black, linkcolor = red, urlcolor = darkblue]{hyperref} 

\theoremstyle{definition} 
\newtheorem{definition}{Definition}[section] 
\theoremstyle{plain} 
\newtheorem{theorem}[definition]{Theorem}

\newtheorem{lemma}[definition]{Lemma}
\newtheorem{corollary}[definition]{Corollary}
\newtheorem{remark}[definition]{Remark}

\allowdisplaybreaks 

\newcommand{\R}{\mathbb{R}} 
\newcommand{\Z}{\mathbb{Z}} 
\newcommand{\ie}{{\it{i.e.}, }} 
 
\newcommand{\vareps}{\varepsilon} 
\newcommand{\del}{\partial} 
\newcommand{\<}{\langle} 
\renewcommand{\>}{\rangle} 
\renewcommand{\O}{\mathcal{O}} 
\newcommand{\tb}{\textbf} 
\newcommand{\tr}{\textnormal}

\renewcommand{\u}{\mathbf{u}} 
\newcommand{\os}{\overset{\circ}}
\newcommand{\od}{\overset{\bullet}} 

\newcommand{\D}{\mathcal{D}}  				
\newcommand{\B}{\mathcal{B}} 				
\newcommand{\F}{\mathcal{F}} 				
\newcommand{\bond}{\mathrm{Bo}} 				
\newcommand{\n}{\mathbf{\hat{n}}}  				
\renewcommand{\H}{\mathbf{H}} 				
\newcommand{\X}{\mathbf{X}} 					
\newcommand{\M}{\mathcal{M}}  				
\newcommand{\Lb}{\mathbf{L}}					
\newcommand{\Mb}{\mathbf{M}} 				
\newcommand{\1}{\boldsymbol{1}} 				
\newcommand{\0}{\boldsymbol{0}} 		 		
\newcommand{\C}{\mathcal{C}} 				
\newcommand{\RR}{\mathcal{R}} 				

\DeclareMathOperator{\PC}{PC}

\usepackage[backend = bibtex, giveninits = true, style = alphabetic, natbib = true, url = false, sorting = nyt, doi = true, abbreviate = true, backref = false]{biblatex}
\renewbibmacro{in:}{\ifentrytype{article}{}{\printtext{\bibstring{in}\intitlepunct}}}

\begin{document}
\title{An Isoperimetric Sloshing Problem in a \\ Shallow Container with Surface Tension} 
\thanks{The work of C.H. Tan and B. Osting was partially funded by NSF DMS 17-52202} 
\author{Chee Han Tan}
\address{Department of Mathematics, Wake Forest University, Winston-Salem, NC} 
\email{tanch@wfu.edu} 

\author{Christel Hohenegger} 
\address{Department of Mathematics, University of Utah, Salt Lake City, UT} 
\email{choheneg@math.utah.edu} 

\author{Braxton Osting} 
\address{Department of Mathematics, University of Utah, Salt Lake City, UT} 
\email{osting@math.utah.edu} 


\date{\today}

\begin{abstract} 
In 1965, B. A. Troesch solved the isoperimetric sloshing problem of determining the container shape that maximizes the fundamental sloshing frequency among two classes of shallow containers: symmetric canals with a given free surface width and cross-sectional area, and radially symmetric containers with a given rim radius and volume \href{https://doi.org/10.1002/cpa.3160180124}{[doi:10.1002/cpa.3160180124]}. Here, we extend these results in two ways: (i) we consider surface tension effects on the fluid free surface, assuming a flat equilibrium free surface together with a pinned contact line, and (ii) we consider sinusoidal waves traveling along the canal with wavenumber $\alpha\ge 0$ and spatial period $2\pi/\alpha$; two-dimensional sloshing corresponds to the case $\alpha = 0$. Generalizing our recent variational characterization of fluid sloshing with surface tension to the case of a pinned contact line, we derive the pinned-edge linear shallow sloshing problem, which is an eigenvalue problem for a generalized Sturm-Liouville system. In the case without surface tension, we show that the optimal shallow canal is a rectangular canal for any $\alpha > 0$. In the presence of surface tension, we solve for the maximizing cross-section explicitly for shallow canals with any given $\alpha\ge 0$ and shallow radially symmetric containers with $m$ azimuthal nodal lines, $m = 0, 1$. Our results reveal that the squared maximal sloshing frequency increases considerably as surface tension increases. Interestingly, both the optimal shallow canal for $\alpha = 0$ and the optimal shallow radially symmetric container are not convex. As a consequence of our explicit solutions, we establish convergence of the maximizing cross-sections, as surface tension vanishes, to the maximizing cross-sections without surface tension. 
\end{abstract} 

\maketitle


\section{Introduction}
Sloshing dynamics refers to the study of the motion of a liquid free surface (\ie the interface between the liquid in the container and the air above) inside partially filled containers or tanks \cite{Faltinsen,Ibrahim}. Liquid sloshing has attracted considerable attention from engineers, scientists, and mathematicians. It is an inevitable phenomenon in many engineering applications, causing detrimental impacts on the dynamics and stability of marine, road, rail, and space transportation systems. For trucks and trains transporting oil and hazardous material, liquid sloshing can affect vehicle dynamics during braking maneuvers and curve negotiation, which could reduce the braking efficiency and increase the risk of vehicle rollover \cite{Vera:2005simulation,Kolaei:2014effects,Otremba:2018modelling}. For liquid propellant spacecraft, violent fuel sloshing produces highly localized pressure on tank walls, leading to deviation from its planned flight path or compromising its structural integrity.  

It is of practical interest to predict the natural sloshing frequencies of the liquid in partially filled containers of arbitrary shape, since large amplitude sloshing tends to occur in the vicinity of resonance, \ie when the external excitation (forcing) frequency of the container is close to one of these natural sloshing frequencies. Knowing these natural frequencies is therefore essential in the analysis and design of liquid containers. In this case, it suffices to consider the linear sloshing problem since the details of the fluid motion are not required in determining the natural frequencies \cite{mciver:1993trough}. Except for very few simple geometries (such as upright cylindrical and rectangular containers) with a flat free surface, where the linear problem has closed-form solutions \cite{Ibrahim,Lamb}, computing the natural sloshing frequencies of a liquid in arbitrarily-shaped containers remains an intricate task but can be treated using a combination of analytical and numerical techniques. Some of the well-known approaches include: 
(1) variational formulations \cite{lawrence:1958,moiseev:1964,moiseev:1966,reynolds:1966,Myshkis,kuznetsov:1990,li:2014,tan:2017}, 
(2) integral equation/conformal mapping \cite{budiansky:1960,chu:1964,fox:1983,fontelos:2020}, 
(3) special coordinate systems \cite{mciver:1989sloshing,mciver:1993trough,shankar:2003}, and 
(4) the series expansion method \cite{evans:1993,shankar:2003,shankar:2007}. 
In this paper, we use a variational approach to study the linear sloshing problem with pinned-end edge constraint.


\subsection{Pinned-edge linear sloshing problem} 
We begin by describing free oscillations of an incompressible, inviscid fluid in a three-dimensional rigid container; a physical derivation of the linear sloshing problem can be found in, {\it e.g.}, \cite[Appendix A]{tan:2017}. Let $\ell_c$ denote a characteristic length scale of the container and $g$ the gravitational acceleration. We nondimensionalize all lengths by $\ell_c$, time by $t_c\coloneqq\sqrt{\ell_c/g}$, and velocity by $\ell_c/t_c$. The fluid occupies a bounded, simply-connected moving domain $\D(t)\subset\R^3$ that is bounded above by the fluid free surface $\F(t)$ and below by piecewise smooth wetted container walls $\B(t)\coloneqq\del\D(t)\setminus\overline{\F(t)}$. Let $(x, y, z)$ be dimensionless Cartesian coordinates such that the  $z$-axis is directed vertically upward. We make the following assumptions:
\begin{enumerate}
\item The fluid flow is \emph{irrotational}. This implies the existence of a \emph{velocity potential} $\phi(x, y, z, t)$ whose gradient is the fluid velocity field, $\u = \nabla\phi$. 
\item The fluid is acted upon by the gravitational force in the bulk and \emph{capillary (surface tension) forces} on the free surface $\F(t)$. The \emph{equilibrium free surface $\F$ is flat} and lies in the plane $z = 0$. The moving free surface can be described by the graph of a function $\eta(x, y, t)$: $\F(t) = \{(x, y, z)\in\R^3\colon z = \eta(x, y, t)\}$. 
\item The \emph{contact line} $\del\F(t)$, \ie the curve at the intersection of the free surface and the container wall, is \emph{fixed at all time}. This translates to $\del_t\eta = 0$ on $\del\F(t)\equiv\del\F$ and is known as the \emph{pinned-end edge constraint}. This was first suggested by Benjamin and Scott \cite{benjamin:1979gravity} to avoid certain discrepancies in the experimental measurements of wave propagation of clean water in a channel. It was subsequently investigated by several authors \cite{graham:1983new,benjamin:1985long,shen:1983nonlinear,henderson:1994,groves:1995theoretical,nicolas:2005effects,shankar:2007,kidambi:2009meniscus}. The pinned contact line has been observed in small amplitude sloshing, and this effect is enhanced on a brimful container or if the fluid exhibits strong surface tension \cite{benjamin:1979gravity,bauer:1992liquid}. 
\item The fluid undergoes \emph{small amplitude oscillations}, allowing for  the linearization of the governing equations around the equilibrium solution $(\phi_0, \eta_0) = (\1, \0)$, \ie write 
$(\phi, \eta) = (\1, \0) + \vareps (\hat\phi,\hat\eta )  + \O(\vareps^2)$, for some ordering parameter $\vareps\ll 1$. 
\end{enumerate} 
This last assumption crucially permits the transformation of the nonlinear fluid problem on the moving domain $\D(t)$ to a linear one (the $\O(\vareps)$ equations) on the fixed, equilibrium domain $\D$. Finally, we seek time-harmonic solutions to the linear problem with natural sloshing frequency $\omega$ and natural sloshing modes $(\Phi, \xi)$, \ie  $\hat\phi(x, y, z, t) = \Phi(x, y, z)\cos(\omega t)$ and $\hat\eta(x, y, t) = \xi(x, y)\sin(\omega t)$. The $\O(\vareps)$ equations (with the time-harmonic factor canceled) are a dimensionless linear boundary spectral problem for $(\omega, \Phi, \xi)$, which we refer to as the \emph{pinned-edge linear sloshing problem}: 
 \begin{subequations} \label{eq:SloshST} 
\begin{alignat}{3}
\label{eq:SloshST1} \Delta\Phi & = 0 && \ \ \tr{ in } \ \ && \D, \\ 
\label{eq:SloshST2} \del_\n\Phi & = 0 && \ \ \tr{ on } \ \ && \B, \\ 
\label{eq:SloshST3} \del_z\Phi & = \omega\xi && \ \ \tr{ on } \ \ && \F, \\ 
\label{eq:SloshST4} \xi - \frac{1}{\bond}\Delta_\F\xi & = \omega\Phi && \ \ \tr{ on } \ \ && \F, \\ 
\label{eq:SloshST5} \xi & = 0 && \ \ \tr{ on } \ \ && \del\F. 
\end{alignat}
\end{subequations}
Here, $\n$ is the outward unit normal vector and $\Delta_\F\xi\coloneqq\del_{xx}\xi + \del_{yy}\xi$ is twice the linearized mean-curvature operator. The dimensionless number $\bond = \rho g\ell_c^2/\sigma$ (with $\rho > 0$ the constant fluid density and $\sigma > 0$ the surface tension coefficient along the free surface) is known as the Bond number and measures the relative magnitudes of gravitational and capillary forces. The mass of the fluid is conserved for any $\omega$. Indeed, a trivial eigenpair of \eqref{eq:SloshST} is given by $(\omega, \Phi, \xi) = (0, \1, \0)$, and for $\omega\neq 0$ mass conservation $\int_\F \xi\, dA = 0$ follows from the divergence theorem.


\subsection{Isoperimetric sloshing problem} 
In the absence of surface tension (\ie $\bond = \infty$), B. A. Troesch \cite{troesch:1965} studied the \emph{isoperimetric sloshing problem} of determining universal upper bounds for sloshing frequencies for the following two families of \emph{shallow symmetric containers}: 
\begin{enumerate}
\item Canals (uniform horizontal channels of arbitrary cross-section) with given free surface width and cross-sectional area; see Figure~\ref{fig:container}(left).  
\item Radially symmetric containers with given rim radius and volume; see Figure~\ref{fig:container}(right). 
\end{enumerate} 
The term \emph{shallow} refers to the assumption that the fluid depth, $h$, is sufficiently small compared to the wavelength of the free oscillation. Exploiting translational symmetry along the canal length for canals and rotational symmetry for radially symmetric containers and then applying the shallow water theory, the three-dimensional sloshing problem in the fluid domain reduces to a one-dimensional problem on the fluid free surface. The isoperimetric sloshing problem becomes a one-dimensional eigenvalue optimization problem with an area or volume constraint, in the sense that the free surface is fixed and only the wetted bottom of the container is allowed to vary. 

\begin{figure}[t] 
\begin{center} 
	\includegraphics[width = 0.645\textwidth]{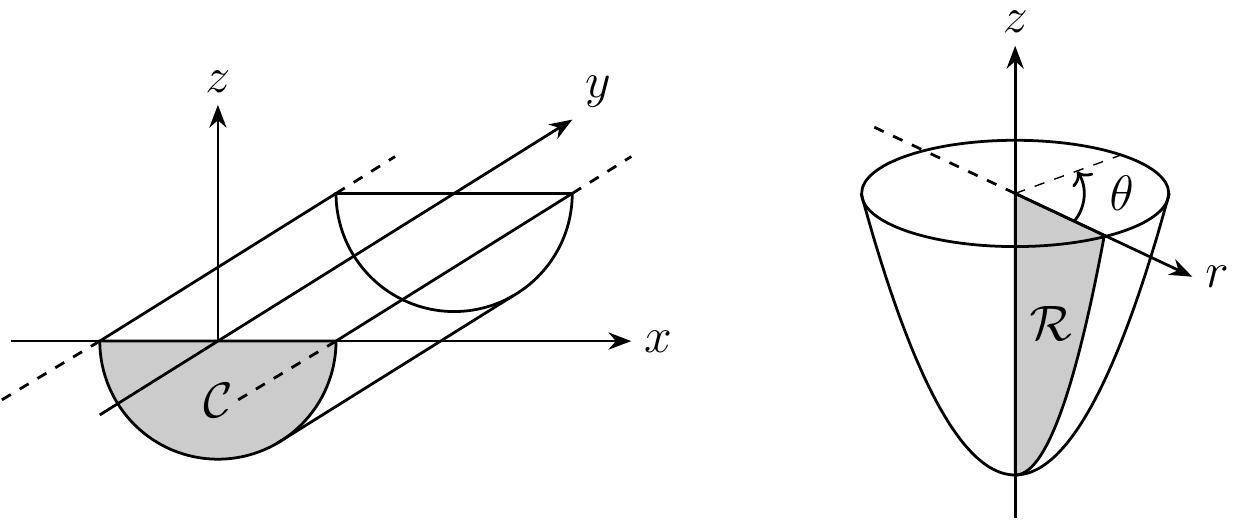} 
\caption{\tb{(Left)} A canal generated by a vertical cross-section $\C$ with container shape $z = -h(x)$. \tb{(Right)} A radially symmetric container obtained as a rotation of a planar meridian domain $\RR$ with container shape $z = -h(r)$.} 
\label{fig:container} 
\end{center} 
\end{figure} 

In the absence of surface tension, Troesch derived the first-order optimality condition using a variational argument, which says that the velocity potential must be linear for any optimal container. Combining this with the extremal property of the sloshing frequencies, Troesch showed that solving the isoperimetric sloshing problem is equivalent to solving a certain first-order singular ODEs with the area or volume constraint. We now state Troesch's result. Let $\lambda_{1, \bond = \infty}$ be the square of the dimensionless fundamental (smallest positive) sloshing frequency without surface tension for a shallow container. For shallow canals with a dimensionless cross-sectional area $A$, we have 
\[ \lambda_{1, \bond = \infty}(h)\le 3A/2 \] 
for all depth functions $h$ within a particular class and the container that saturates this inequality is a parabola; in particular, the maximizing cross-section is convex and has no vertical side walls. For shallow radially symmetric containers with a dimensionless volume $V$, $\lambda_{1, \bond = \infty}$ (corresponding to a motion with one nodal diameter $m = 1$ on the free surface) satisfies 
\[ \lambda_{1, \bond = \infty}(h)\le 4V/\pi \] 
for all depth functions $h$ within a particular class and the container that saturates this inequality is again a parabola. The problem for higher sloshing frequencies is solved numerically and these optimal containers are not connected, in the sense that the container depth vanishes at some point in the interior of the free surface.


\subsection{Summary of main results and outline} 
The goal of this work is twofold: (i) solve the isoperimetric sloshing problem for the fundamental (smallest positive) sloshing frequency, including the effects of surface tension on the free surface and (ii) extend transverse sloshing to longitudinal sloshing along the canal. We now give precise statements of our main results. \\ 

\noindent\tb{Canals.} In the case of canals, we restrict our attention to sinusoidal solutions of the form $\Phi(x, y, z) = \varphi(x, z)\cos(\alpha y)$ and $\xi(x, y) = \zeta(x)\cos(\alpha y)$, where 
the parameter $\alpha\ge 0$ is the wavenumber associated with the longitudinal sloshing mode with spatial period $2\pi/\alpha$. In particular, the case $\alpha = 0$ corresponds to planar sloshing in the vertical $xz$-plane. For a (sufficiently small) fixed cross-sectional area $A > 0$, we introduce the following class of admissible shape functions for shallow canals 
\[ \M_A \coloneqq \left\{h\in \PC^1[-1, 1]\colon \tr{$h \ge 0$ on $[-1, 1]$; $\int_{-1}^1 h\, dx = A$}\right\}, \] 
where $\PC^1[-1, 1]$ denotes the set of all continuous and piecewise continuously differentiable functions on the interval $[-1, 1]$. For every $\alpha\ge 0$, let $\Omega_{\alpha, 1}^\infty$ and $\Omega_{\alpha ,1}$ denote the fundamental sloshing frequency for a shallow canal in the absence ($\bond = \infty$) and presence ($\bond < \infty$) of surface tension, respectively. Define $\lambda_{\alpha, 1}^\infty\coloneqq (\Omega_{\alpha, 1}^\infty)^2$ and $\lambda_{\alpha, 1}\coloneqq \Omega_{\alpha, 1}^2$. Troesch proved that the parabolic cross-section maximizes $\lambda_{0, 1}^\infty$ in $\M_A$, with upper bound $\lambda_{0, 1}^{\infty, *}\coloneqq 3A/2$. Our first theorem extends Troesch's result from $\alpha = 0$ to $\alpha > 0$ in the absence of surface tension. 

\begin{theorem}[$\bond = \infty$, $\alpha > 0$] \label{thm:Iso2D_zero} 
Let $h$ be a shape function in $\M_A$. Then for all $\alpha > 0$, we have 
\begin{equation*} \label{eq:Iso2D_zero_eigs} 
\lambda_{\alpha, 1}^\infty(h)\le \lambda_{\alpha, 1}^{\infty, *}\coloneqq \frac{\alpha^2A}{2}. 
\end{equation*} 
Equality holds for $h = h_\alpha^{\infty, *}\coloneqq A/2$, \ie the maximizing cross-section is a rectangle for any $\alpha > 0$.  
\end{theorem} 

Observe that $\alpha\mapsto\lambda_{\alpha, 1}^{\infty, ^*}$ is strictly increasing on $(0, \infty)$ and $\lambda_{\alpha, 1}^{\infty, *}\to 0\neq\lambda_{0, 1}^{\infty, *}$ as $\alpha\to 0^+$. We associate this with the fact that the trivial eigenvalue $\lambda = 0$ exists for the shallow sloshing problem without surface tension \eqref{eq:SloshCanalS_zero} for $\alpha = 0$ but not for $\alpha > 0$. 

The next two theorems generalize Troesch's result for $\alpha = 0$ and Theorem \ref{thm:Iso2D_zero} for $\alpha > 0$ from $\bond = \infty$ to $\bond < \infty$. 

\begin{theorem}[$\bond < \infty$, $\alpha = 0$] \label{thm:Iso2D_0} 
Let $h$ be a shape function in $\M_A$. Then the following inequality holds: 
\begin{equation} \label{eq:Iso2D_0_eigs} 
\lambda_{0, 1}(h)\le \lambda_{0, 1}^*\coloneqq \frac{3A}{2}\left[1 - \frac{3\left(\sqrt{\bond} - \tanh\left(\sqrt{\bond}\, \right)\right)}{\bond\tanh\left(\sqrt{\bond}\, \right)}\right]^{-1}. 
\end{equation} 
Equality holds for $h = h_0^*$ defined by 
\begin{equation} \label{eq:Iso2D_0_shape} 
h_0^*(x) = \frac{\lambda_{0, 1}^*}{2}(1 - x^2) - \frac{\lambda_{0, 1}^*}{\sqrt{\bond}\sinh(\sqrt{\bond}\, )}\left[\cosh(\sqrt{\bond}\, ) - \cosh(\sqrt{\bond}\, x)\right]. 
\end{equation} 
In particular, $z = -h_0^*$ is symmetric but not convex on $[-1, 1]$. 
\end{theorem} 

\begin{theorem}[$\bond < \infty$, $\alpha > 0$]\label{thm:Iso2D_1} 
Define $\kappa_\alpha\coloneqq \sqrt{\alpha^2 + \bond}$. Let $h$ be a shape function in $\M_A$. Then for all $\alpha > 0$, we have 
\begin{equation} \label{eq:Iso2D_1_eigs} 
\lambda_{\alpha, 1}(h)\le \lambda_{\alpha, 1}^*\coloneqq \frac{\alpha^2A}{2}\frac{\kappa_\alpha^2}{\bond}\left[1 - \frac{\tanh\kappa_\alpha}{\kappa_\alpha}\right]^{-1}. 
\end{equation} 
Equality holds for $h = h_\alpha^*$ defined by 
\begin{equation} \label{eq:Iso2D_1_shape} 
h_\alpha^*(x) = \frac{\lambda_{\alpha, 1}^*\bond}{\alpha^2\kappa_\alpha^2}\left[1 - \frac{\cosh\left(\kappa_\alpha x\right)}{\cosh\kappa_\alpha}\right]. 
\end{equation} 
In particular, $z = -h_\alpha^*$ is symmetric and convex on $[-1, 1]$. 
\end{theorem} \phantom{x} 

\noindent\tb{Radially symmetric containers.} In the case of radially symmetric containers, we convert to cylindrical coordinates $(r, \theta, z)$ and look for solutions of the form $\Phi(r, \theta, z) = \varphi(r, z)\cos(m\theta)$ and $\xi(r, \theta) = \zeta(r)\cos(m\theta)$, with $m = 0, 1, 2, \dots$ the number of azimuthal nodal lines. For a (sufficiently small) fixed volume $V > 0$, we introduce the following class of admissible shape functions for shallow radially symmetric containers 
\[ \M_V\coloneqq \left\{h\in \PC^1[0, 1]\colon \tr{$h\ge 0$ on $[0, 1]$; $\int_0^1 hr\, dr = V/2\pi$}\right\}. \] 
For every $m = 0, 1, 2, \dots$, let $\Omega_{m, 1}$ denote the fundamental sloshing frequency for a shallow radially symmetric container in the presence of surface tension and define $\lambda_{m, 1}\coloneqq \Omega_{m, 1}^2$. The next two theorems generalize Troesch's result for $m = 1$ and $m = 0$ from $\bond = \infty$ to $\bond < \infty$. Throughout this paper, $I_\nu$ and $\Lb_\nu$ are modified Bessel and Struve functions of the first kind of order $\nu$, respectively, and ${}_pF_q$ is the generalized hypergeometric function; see \cite[Section~10.25]{NIST} for $I_\nu$, \cite[Chapter~11]{NIST} for $\Lb_\nu$, and \cite[Chapter~16]{NIST} for ${}_pF_q$. 

\begin{theorem}[$m = 1$] \label{thm:Iso3D_1} 
Let $h$ be a shape function in $\M_V$. Then the following inequality holds: 
\begin{equation} \label{eq:Iso3D_1_eigs} 
\lambda_{1, 1}(h)\le \lambda_{1, 1}^*\coloneqq \frac{4V}{\pi}\left[1 - \frac{4I_2(\sqrt{\bond}\, )}{\sqrt{\bond}\, I_1(\sqrt{\bond}\, )}\right]^{-1}. 
\end{equation} 
Equality holds for $h = h_1^*$ defined by 
\begin{equation} \label{eq:Iso3D_1_shape} 
h_1^*(r) = \frac{\lambda_{1, 1}^*}{2}(1 - r^2) - \frac{\lambda_{1, 1}^*}{\sqrt{\bond}\, I_1(\sqrt{\bond}\, )}\left[I_0(\sqrt{\bond}\, ) - I_0(\sqrt{\bond}\, r)\right]. 
\end{equation} 
In particular, $z = -h_1^*$ is not convex on $[0, 1]$. 
\end{theorem} 

\begin{theorem}[$m = 0$] \label{thm:Iso3D_0} 
Let $h$ be a shape function in $\M_V$. Then the following inequality holds: 
\begin{equation} \label{eq:Iso3D_0_eigs} 
\begin{alignedat}{1} 
\lambda_{0, 1}(h) & \le \lambda_{0, 1}^*\coloneqq \frac{18V}{\pi}\bigg[6d_0 - 3 + \frac{18(1 - d_0)\pi\Upsilon(\sqrt{\bond}\, )}{\bond\, I_0(\sqrt{\bond}\, )} \\ 
& \hspace{2.5cm} + 3{}_2F_3\left(1, 2; \frac{3}{2}, \frac{5}{2}, 3; \frac{\bond}{4}\right) - \frac{9\pi^2\Lb_0(\sqrt{\bond}\, )\Upsilon(\sqrt{\bond}\, )}{\bond^{3/2}I_0(\sqrt{\bond}\, )}\bigg]^{-1},  
\end{alignedat} 
\end{equation}
where 
\begin{equation} \label{eq:Iso3D_0_d0} 
\begin{alignedat}{1} 
\Upsilon(\sqrt{\bond}\, ) & \coloneqq I_1(\sqrt{\bond}\, )\, \Lb_0(\sqrt{\bond}\, ) - I_0(\sqrt{\bond}\, )\, \Lb_1(\sqrt{\bond}\, ) \\ 
d_0 & \coloneqq \frac{3\pi\Upsilon(\sqrt{\bond}\, ) + 2\bond\, I_0(\sqrt{\bond}\, ) - 6\sqrt{\bond}\, I_1(\sqrt{\bond}\, )}{3\bond\, I_2(\sqrt{\bond}\, )}. 
\end{alignedat} 
\end{equation} 
Equality holds for $h = h_0^*$ defined by 
\begin{equation} \label{eq:Iso3D_0_shape} 
\begin{alignedat}{1} 
h_0^*(r) = \frac{\lambda_{0, 1}^*}{3}\left[\frac{3d_0r}{2} - r^2\right] & + \frac{\lambda_{0, 1}^*(1 - d_0)I_1(\sqrt{\bond}\, r)}{\sqrt{\bond}\, I_0(\sqrt{\bond}\, )} \\ 
& + \frac{\lambda_{0, 1}^*\pi}{2\bond}\left[\Lb_1(\sqrt{\bond}\, r) - \frac{\Lb_0(\sqrt{\bond}\, )I_1(\sqrt{\bond}\, r)}{I_0(\sqrt{\bond}\, )}\right]. 
\end{alignedat} 
\end{equation} 
In particular, $h_0^*(0) = 0$ and $z = -h_0^*$ is not convex on $[0, 1]$. 
\end{theorem} \phantom{x} 

\noindent \tb{Effects of surface tension.} Let us describe the effects of surface tension ($\bond < \infty$) on the solution to the isoperimetric sloshing problem. 
\begin{enumerate} 
\item Except for the case $\alpha > 0$, the optimal shallow containers $h^*$ are no longer convex. Specifically, these containers flatten near the contact point, \ie $h^*$ and its derivative vanish on the boundary of the free surface $\F$.  
\item Our isoperimetric inequalities for $\lambda_{1, \bond}(h)$ can be interpreted as 
\begin{equation*} 
\lambda_{1, \bond}(h) \le \lambda_{1, \bond }^* = C \cdot\lambda_{1, \bond = \infty}^*, 
\end{equation*} 
where $C = C(\bond) > 1$ for $\bond > 0$ and $\lambda_{1, \bond = \infty}^*$ is the squared maximal sloshing frequency for $\bond = \infty$. In particular, $\bond\mapsto C(\bond)$ is strictly decreasing on $(0, \infty)$ and for $\bond = 1$, $C$ is approximately $16.4$ for $\alpha = 0$, between 4.2 and 48 for $\alpha < 2\pi$ say, 25.5 for $m = 1$, and 38.8 for $m = 0$, \ie \emph{surface tension results in a significantly larger squared maximal sloshing frequency}. 
\end{enumerate} 
Because all our solutions are explicit, we are able to show that the limit of the solution $(h_{\bond}^*, \lambda_{1, \bond}^*)$ to the isoperimetric sloshing problem with surface tension, as surface tension vanishes, \ie $\bond\to\infty$, is the solution $(h_{\bond = \infty}^*, \lambda_{1, \bond = \infty}^*)$ to the isoperimetric sloshing problem without surface tension; see Corollaries \ref{thm:Iso2D_zeroST_0}, \ref{thm:Iso2D_zeroST_1}, \ref{thm:Iso3D_zeroST_1}, \ref{thm:Iso3D_zeroST_0}. 

The results in this paper lend theoretical justification for the practice of adding surfactant to a liquid in certain vessel geometries to change the fundamental sloshing frequency and help mitigate negative consequences of sloshing dynamics in certain applications. 

\begin{remark} 
There are a few assertions in the proofs of our main results that we verify numerically. Except for Theorem \ref{thm:Iso2D_zero}, we verify using the finite difference method that $\lambda_{1, \bond}^*$ is indeed the squared \emph{fundamental} sloshing frequency for the corresponding shape function $h_{1, \bond}^*$. For Corollary \ref{thm:Iso3D_zeroST_0}, we use Mathematica to compute the limit of the third term of $h_0^*$ involving combinations of the modified Bessel and Struve functions as $\bond\to\infty$ for any $r\in (0, 1)$. 
\end{remark}

\noindent\tb{Outline.} The paper is structured as follows. In sections \ref{sec:canal} and \ref{sec:RS}, we describe the reduction of \eqref{eq:SloshST} when the container is a canal and a radially symmetric container, respectively. We generalize our recent variational characterization of fluid sloshing with surface tension \cite{tan:2017} to the case of a pinned contact line in subsections \ref{sec:canal_variational} and \ref{sec:RS_variational}. In subsections \ref{sec:Iso2D} and \ref{sec:Iso3D}, we derive the pinned-edge linear shallow sloshing problem and outline our approach in solving the isoperimetric sloshing problem. We prove Theorem \ref{thm:Iso2D_zero} in subsection \ref{sec:Iso2D_zero}, Theorems \ref{thm:Iso2D_0} and \ref{thm:Iso2D_1} in subsection \ref{sec:Iso2D_OC}, and Theorems \ref{thm:Iso3D_1} and \ref{thm:Iso3D_0} in subsection \ref{sec:Iso3D_OC}. We establish the zero surface tension limit of the solution to the isoperimetric sloshing problem in subsections \ref{sec:Iso2D_zeroST} and \ref{sec:Iso3D_zeroST}. We conclude in Section \ref{sec:disc} with a discussion.


\section{Canals} \label{sec:canal} 
We choose the halfwidth $x_0$ of the equilibrium free surface as the characteristic length scale. Let $(x, y, z)$ be dimensionless Cartesian coordinates with $y$ directed along the length of the canal, which is unbounded and $z$ vertically upwards; see Figure \ref{fig:container}(left). Let $h(x)$ be a dimensionless function describing the profile of the container's bottom. A canal is the equilibrium fluid domain $\D = \C\times\{y\in\R\}$, bounded by the wetted bottom $\B = (-1, 1)\times\R\times\{z = -h(x)\}$, the free surface $\F = (-1, 1)\times\R\times\{0\}$, and the contact line $\del\F = \{\pm 1\}\times\R\times\{0\}$. We think of $\C$ as the cross-section of $\D$ that lies on the plane $y = 0$. 

We seek sinusoidal solutions of \eqref{eq:SloshST} oscillating with wavenumber $\alpha\ge 0$ in the positive $y$-direction, \ie we write the natural sloshing modes $(\Phi, \xi)$ as $\Phi(x, y, z) = \varphi(x, y)\cos(\alpha y)$ and $\xi(x, y) = \zeta(x)\cos(\alpha y)$. These ansatzes reduce \eqref{eq:SloshST} to the following two-dimensional generalized mixed Steklov problem for $(\omega, \varphi, \zeta)$ on $\overline{\C}$: 
\begin{subequations} \label{eq:SloshCanal} 
\begin{alignat}{3} 
\label{eq:SloshC1} \Delta_\C\varphi & = \alpha^2\varphi && \ \ \tr{ in } \ \ && \C, \\
\label{eq:SloshC2} \del_\n\varphi & = 0 && \ \ \tr{ on } \ \ && \B_\C\coloneqq\del\C\cap\B, \\ 
\label{eq:SloshC3} \del_z\varphi & = \omega\zeta && \ \ \tr{ on } \ \ && \F_\C\coloneqq\del\C\cap\F, \\ 
\label{eq:SloshC4} \left(1 + \frac{\alpha^2}{\bond}\right)\zeta - \frac{1}{\bond}\zeta'' & = \omega\varphi && \ \ \tr{ on } \ \ && \F_\C, \\
\label{eq:SloshC5} \zeta(\pm 1) & = 0. 
\end{alignat} 
\end{subequations} 
Here, $\nabla_\C\coloneqq (\del_x, \del_z)$, $\Delta_\C = \nabla_\C\cdot\nabla_\C = \del_{xx} + \del_{zz}$, and we now have two \emph{contact points} $(x, z) = (\pm 1, 0)$ in \eqref{eq:SloshCanal}. It is straightforward to verify that the ansatz for $\xi$ satisfies the necessary condition $\int_\F\xi\, dA = 0$ due to the factor $\cos(\alpha y)$ for $\alpha > 0$. The case $\alpha = 0$ corresponds to the two-dimensional transverse sloshing problem on $\overline{\C}$ and a necessary condition for the existence of solutions of \eqref{eq:SloshCanal} is $\int_{-1}^1 \zeta\, dx = 0$.


\subsection{Variational principle} \label{sec:canal_variational} 
Let $H^1(\C) = W^{1, 2}(\C)$ and $H_0^1(-1, 1) = W_0^{1, 2}(-1, 1)$ denote the standard Sobolev spaces with real-valued functions, both equipped with norms induced by their natural inner products. Define the Hilbert space $\H_\C\coloneqq H^1(\C)\times H_0^1(-1, 1)$ with norm $\|(\varphi, \zeta)\|_{\H_\C}^2\coloneqq \|\varphi\|_{H^1(\C)}^2 + \|\zeta\|_{H^1(-1, 1)}^2$. Suppose $(\omega, \varphi, \zeta)$ is a sufficiently regular solution of \eqref{eq:SloshCanal} for $\alpha > 0$. Testing \eqref{eq:SloshC1}, \eqref{eq:SloshC4} with $f\in H^1(\C), g\in H_0^1(-1, 1)$, respectively, and using the remaining equations in \eqref{eq:SloshCanal}, we arrive at the following weak formulation of \eqref{eq:SloshCanal} for $\alpha > 0$. 

\begin{definition} 
Given $\alpha > 0$, we say that $(\omega_\alpha, \varphi_\alpha, \zeta_\alpha)\in\R\times\H_\C, (\varphi_\alpha, \zeta_\alpha)\neq (\0, \0)$ is a weak sloshing eigenpair of \eqref{eq:SloshCanal} if the following holds for all $(f, g)\in\H_\C$: 
\begin{equation} \label{eq:Slosh2D_weak} 
\begin{alignedat}{1} 
\int_\C \left(\nabla_C\varphi_\alpha\cdot\nabla_\C f + \alpha^2\varphi_\alpha f\right) dA & + \int_{-1}^1 \left[\left(1 + \frac{\alpha^2}{\bond}\right)\zeta_\alpha g + \frac{1}{\bond}\zeta_\alpha'g'\right] dx \\ 
& = \omega_\alpha\int_{\F_\C} \left(\zeta_\alpha f + \varphi_\alpha g\right) dx. 
\end{alignedat} 
\end{equation} 
\end{definition} 

For $\alpha = 0$, we introduce the Hilbert space $\H_{\C, 0}\coloneqq \{(\varphi, \zeta)\in\H_\C\colon \int_{-1}^1 \zeta\, dx = 0\}$ which is a closed subspace of $\H_\C$, with norm induced by the norm of $\H_\C$. 

\begin{definition} 
We say that $(\omega_0, \varphi_0, \zeta_0)\in\R\times\H_{\C, 0}, (\varphi_0, \zeta_0)\neq (\0, \0)$ is a weak sloshing eigenpair of \eqref{eq:SloshCanal} for $\alpha = 0$ if the following holds for all $(f, g)\in\H_{\C, 0}$: 
\begin{equation} \label{eq:Slosh2D_weak0} 
\int_\C \nabla_C\varphi_0\cdot\nabla_\C f\, dA + \int_{-1}^1 \left(\zeta_0 g + \frac{1}{\bond}\zeta_0'g'\right) dx = \omega_0\int_{\F_\C} \left(\zeta_0 f + \varphi_0 g\right) dx. 
\end{equation} 
\end{definition} 

If $(\omega_\alpha, \varphi_\alpha, \zeta_\alpha)$ is a weak sloshing eigenpair of \eqref{eq:SloshCanal}, then so are $(-\omega_\alpha, \pm\varphi_\alpha, \mp\zeta_\alpha)$ and we may restrict our attention to weak sloshing eigenpairs with $\omega_\alpha > 0$. We now derive a sufficient condition for obtaining positive sloshing frequencies. To this end, define the functional 
\[ G_\C(\varphi, \zeta) = \int_{\F_\C} \varphi\zeta\, dx, \] 
and the energy functional $E_{\C, \alpha}(\varphi, \zeta)$ which is the sum of the kinetic energy $D_{\C, \alpha}(\varphi)$ and the potential energy $S_{\C, \alpha}(\zeta)$ of the fluid under small amplitude oscillations: 
\begin{align*} 
E_{\C, \alpha}(\varphi, \zeta) & \coloneqq \underbrace{\frac{1}{2}\int_\C \Big[|\nabla_\C\varphi|^2 + \alpha^2\varphi^2\Big]\, dA}_{D_{\C, \alpha}(\varphi)} + \underbrace{\frac{1}{2}\int_{-1}^1 \left[\left(1 + \frac{\alpha^2}{\bond}\right)\zeta^2 + \frac{1}{\bond}(\zeta')^2\right] dx}_{S_{\C, \alpha}(\zeta)}. 
\end{align*} 

\begin{lemma} \label{thm:Slosh2D_VC1} 
Given $\alpha\ge 0$, suppose $(\omega_\alpha, \varphi_\alpha, \zeta_\alpha)$ is a weak sloshing eigenpair of \eqref{eq:SloshCanal}. We have the identity $\omega_\alpha G_\C(\varphi_\alpha, \zeta_\alpha) =  D_{\C, \alpha}(\varphi_\alpha) + S_{\C, \alpha}(\zeta_\alpha) = E_{\C, \alpha}(\varphi_\alpha, \zeta_\alpha)$. 
In particular, $\omega_\alpha$ and $G_\C(\varphi_\alpha, \zeta_\alpha)$ have the same sign provided $G_\C(\varphi_\alpha, \zeta_\alpha)\neq 0$. 
\end{lemma} 
\begin{proof} 
The identity follows from substituting $(f, g) = (\varphi_\alpha, \zeta_\alpha)$ in \eqref{eq:Slosh2D_weak} and \eqref{eq:Slosh2D_weak0}. The second assertion follows by noting that $E_{\C, \alpha}$ is nonnegative and has a trivial kernel for $\alpha > 0$ and the one-dimensional kernel spanned by $(\varphi, \zeta) = (\1, \0)$ for $\alpha = 0$. 
\end{proof} 

For $\alpha\ge 0$, let $\omega_{\alpha, 1}$ denote the fundamental (smallest positive) sloshing frequency of \eqref{eq:SloshCanal} with corresponding weak fundamental sloshing eigenfunction $(\varphi_{\alpha, 1}, \zeta_{\alpha, 1})$. We are now prepared to establish a variational characterization for $\omega_{\alpha, 1}$ which is inspired by \cite{tan:2017} and Lemma \ref{thm:Slosh2D_VC1}. Note that since $E_{\C, \alpha}$ is a homogeneous functional of degree 2, minimizing $E_{\C, \alpha}(\varphi, \zeta)/G_\C(\varphi, \zeta)$ over all nonzero functions $(\varphi, \zeta)$ satisfying $G_\C(\varphi, \zeta) > 0$ is equivalent to minimizing $E_{\C, \alpha}(\varphi, \zeta)$ over all functions $(\varphi, \zeta)$ satisfying the integral constraint $G_\C(\varphi, \zeta) = 1$. 

\begin{theorem}[Variational characterization, $\alpha > 0$] \label{thm:Slosh2D_VC2} 
Let $\C$ be a bounded Lipschitz domain in $\R^2$. For every $\alpha > 0$, there exists a weak fundamental sloshing eigenpair $(\omega_{\alpha, 1}, \varphi_{\alpha, 1}, \zeta_{\alpha, 1})$ of \eqref{eq:SloshCanal}, where $(\varphi_{\alpha, 1}, \zeta_{\alpha, 1})$ is a constrained minimizer of the following variational problem: 
\begin{equation} \label{eq:Slosh2D_VC2} 
\omega_{\alpha, 1}\coloneqq \inf_{(\varphi, \zeta)\in\H_\C} \left\{E_{\C, \alpha}(\varphi, \zeta)\colon G_\C(\varphi, \zeta) = 1\right\}. 
\end{equation} 
\end{theorem} 
\begin{proof} 
Fix $\alpha > 0$. Define the admissible set $M\coloneqq \left\{(\varphi, \zeta)\in\H_\C\colon G_\C(\varphi, \zeta) = 1\right\}$. The existence of a minimizer to \eqref{eq:Slosh2D_VC2} for $\alpha > 0$ follows from the direct method of the calculus of variations, as $M$ is weakly closed in $\H_\C$ and $E_{\C, \alpha}$ is weakly coercive and weakly lower semicontinuous on $M$ with respect to $\H_\C$; see \cite[Lemmas 3.5 and 3.6]{tan:2017} for similar proofs of these assertions. Let $(\varphi_\alpha^*, \zeta_\alpha^*)$ be a minimizer of \eqref{eq:Slosh2D_VC2}. It is straightforward to verify that both $E_{\C, \alpha}$ and $G_\C$ are continuously Fr\'{e}chet differentiable on $\H_\C$ and for any $(\varphi, \zeta)\in M$ we have 
$\<DG_\C(\varphi, \zeta), (\varphi, \zeta)\>_{\H_\C', \H_\C} = \int_{\F_\C} 2\varphi\zeta\, dx = 2\neq 0, $
where $\H_\C'$ is the dual space of $\H_\C$ and $\<\cdot, \cdot\>_{\H_\C', \H_\C}$ denotes the duality pairing between $\H_\C$ and $\H_\C'$. Thus the Lagrange multiplier rule applies and there exists a Lagrange multiplier $\mu_\alpha\in\R$ such that  
\begin{equation} \label{eq:Slosh2D_VC2a} 
\frac{d}{d\vareps}(E_{\C, \alpha} - \mu_\alpha G_\C)(\varphi_\alpha^* + \vareps f, \zeta_\alpha^* + \vareps g)\bigg|_{\vareps = 0} = 0 \ \ \tr{ for any $(f, g)\in\H_\C$}. 
\end{equation} 
A direct computation shows that \eqref{eq:Slosh2D_VC2a} is equivalent to $(\mu_\alpha, \varphi_\alpha^*, \zeta_\alpha^*)\in\R\times M$ satisfying \eqref{eq:Slosh2D_weak} for all $(f, g)\in\H_\C$. Moreover, Lemma \ref{thm:Slosh2D_VC1} together with $G_\C(\varphi_\alpha^*, \zeta_\alpha^*) = 1$ gives $\mu_\alpha = E_{\C, \alpha}(\varphi_\alpha^*, \zeta_\alpha^*) > 0$. It remains to show that $E_{\C, \alpha}(\varphi_\alpha^*, \zeta_\alpha^*)$ is the fundamental sloshing frequency, but this follows immediately by choosing any weak sloshing eigenfunction $(\varphi_\alpha, \zeta_\alpha)\in M$ as a trial function in \eqref{eq:Slosh2D_VC2} and applying Lemma \ref{thm:Slosh2D_VC1}. 
\end{proof} 

\begin{theorem}[Variational characterization, $\alpha = 0$] \label{thm:Slosh2D_VC3} 
Let $\C$ be a bounded Lipschitz domain in $\R^2$. There exists a weak fundamental sloshing eigenpair $(\omega_{0, 1}, \varphi_{0, 1}, \zeta_{0, 1})$ of \eqref{eq:SloshCanal} for $\alpha = 0$, where $(\varphi_{0, 1}, \zeta_{0, 1})$ is a constrained minimizer of the following variational problem: 
\begin{equation} \label{eq:Slosh2D_VC3} 
\omega_{0, 1}\coloneqq \inf_{(\varphi, \zeta)\in\H_{\C, 0}} \left\{E_{\C, 0}(\varphi, \zeta)\colon G_\C(\varphi, \zeta) = 1\right\}. 
\end{equation} 
\end{theorem} 
\begin{proof} 
Define the admissible set $M\coloneqq \{(\varphi, \zeta)\in\H_{\C, 0}\colon G_\C(\varphi, \zeta) = 1\}$, its subset $N\coloneqq M\cap\{\varphi\colon \int_{\F_\C} \varphi\, dx = 0\}$, and the operator $P\varphi\coloneqq \varphi - \frac{1}{|\F_\C|}\int_{\F_\C} \varphi\, dx$. We claim that a minimizer of \eqref{eq:Slosh2D_VC3} is given by $(\widehat\varphi, \widehat\zeta)$, where $(\widehat\varphi, \widehat\zeta)$ is a constrained minimizer of $E_{\C, 0}$ over $N$. Indeed, since $(P\varphi, \zeta)\in N$ for any $(\varphi, \zeta)\in M$, we obtain 
\begin{align*} 
E_{\C, 0}(\varphi, \zeta) = E_{\C, 0}(P\varphi, \zeta) \ge E_{\C, 0}(\widehat\varphi, \widehat\zeta) \ \ \textrm{ for any $(\varphi, \zeta)\in M$}. 
\end{align*} 
The existence of $(\widehat\varphi, \widehat\zeta)$ follows from the arguments given in \cite[Theorem 1.1]{tan:2017}, in particular we have that $(\widehat\mu, \widehat\varphi, \widehat\zeta)$ with $\widehat\mu = E_{\C, 0}(\widehat\varphi, \widehat\zeta)$ satisfies \eqref{eq:Slosh2D_weak0} for all $(\widehat f, \widehat g)\in\H_{\C, 0}\cap\{\widehat f\colon \int_{\F_\C} \widehat f\, dx = 0\}$. Choosing $(\widehat f, \widehat g) = (Pf, g)$ for any $(f, g)\in \H_{\C, 0}$ and using $\int_{-1}^1 \widehat\zeta\, dx = 0$, we find $(\widehat\mu, \widehat\varphi, \widehat\zeta)$ satisfies \eqref{eq:Slosh2D_weak0} for all $(f, g)\in\H_{\C, 0}$, \ie $(\widehat\mu, \widehat\varphi, \widehat\zeta)$ is a weak sloshing eigenpair of \eqref{eq:SloshCanal} for $\alpha = 0$. Finally, a similar argument from Theorem \ref{thm:Slosh2D_VC2} shows that $\widehat\mu$ is the fundamental sloshing frequency and we omit the proof for brevity. 
\end{proof} 

\begin{remark} 
Although we have a family of minimizers $(\widehat\varphi + c, \widehat\zeta)$ to \eqref{eq:Slosh2D_VC3} for any $c\in\R$, \eqref{eq:SloshC4} shows that $c = -\frac{1}{\omega_{0, 1}\bond}\int_{-1}^1 \widehat\zeta''\, dx$ for sufficiently regular $(\widehat\varphi, \widehat\zeta)$. 
\end{remark}


\subsection{Isoperimetric sloshing problem on shallow canals} \label{sec:Iso2D} 
Let us now restrict our attention to shallow canals. Recall the class of admissible shape functions for shallow canals 
\[ \M_A = \left\{h\in \PC^1[-1, 1]\colon \tr{$h\ge 0$ on $[-1, 1]$; $\int_{-1}^1 h\, dx = A$}\right\}. \] 
Appealing to the shallow water theory, we may assume that $\varphi$ is independent of the depth, \ie $\varphi(x, z) = \psi(x)$. Define the Hilbert space $\H\coloneqq H^1(-1, 1)\times H_0^1(-1, 1)$ and its closed subspace $\H_0\coloneqq \H\cap \{\zeta\colon \int_{-1}^1 \zeta\, dx = 0\}$. Following \cite{lawrence:1958,troesch:1965} and thanks to Theorems \ref{thm:Slosh2D_VC2} and \ref{thm:Slosh2D_VC3}, we may approximate $\omega_{\alpha, 1}$ as the infimum of the following one-dimensional constrained variational problem: 
\begin{equation} \label{eq:Slosh2D_VCs} 
\omega_{\alpha, 1}(h)\approx\Omega_{\alpha, 1}(h)\coloneqq \inf_{(\psi, \zeta)\in\X_\alpha}\left\{E_\alpha(h; \psi, \zeta)\colon G(\psi, \zeta) = 1\right\}, 
\end{equation} 
where $\X_0\coloneqq \H_0$, $\X_\alpha\coloneqq \H$ for every $\alpha > 0$, $G(\psi, \zeta)\coloneqq \int_{-1}^1 \psi\zeta\, dx$, and 
\begin{align*} 
E_\alpha(h; \psi, \zeta) & \coloneqq \underbrace{\frac{1}{2}\int_{-1}^1 h\Big[(\psi')^2 + \alpha^2\psi^2\Big]\, dx}_{D_{\C, \alpha}(\psi) \coloneqq D_\alpha(h; \psi)} + \frac{1}{2}\int_{-1}^1 \left[\left(1 + \frac{\alpha^2}{\bond}\right)\zeta^2 + \frac{1}{\bond}(\zeta')^2\right] dx. 
\end{align*} 

We then define the weak formulation of the one-dimensional pinned-edge shallow sloshing problem for every $\alpha\ge 0$ on $[-1, 1]$ as the weak form of the Euler-Lagrange equations of the constrained variational problem \eqref{eq:Slosh2D_VCs}. 

\begin{definition} \label{def:Slosh2D_weak_s} 
Given $\alpha\ge 0$, a weak shallow sloshing eigenpair $(\Omega_\alpha, \psi_\alpha, \zeta_\alpha)\in\R\times\X_\alpha$ satisfies the following equation for all $(f, g)\in\X_\alpha$: 
\begin{align*} 
\int_{-1}^1 \bigg[h\Big(\psi_\alpha'f' + \alpha^2\psi_\alpha f\Big) + \bigg(1 + \frac{\alpha^2}{\bond}\bigg)\zeta_\alpha g + \frac{1}{\bond}\zeta_\alpha'g'\bigg]\, dx = \Omega_\alpha\int_{-1}^1 (\zeta_\alpha f + \psi_\alpha g)\, dx. 
\end{align*} 
\end{definition} 

The corresponding one-dimensional boundary eigenvalue problem for every $\alpha\ge 0$ is given by 
\begin{subequations}\label{eq:SloshCanalS} 
\begin{alignat}{2} 
\label{eq:SloshCS1} -(h\psi')' + \alpha^2h\psi & = \Omega\zeta && \ \ \tr{ on $(-1, 1)$}, \\ 
\label{eq:SloshCS2} \left(1 + \frac{\alpha^2}{\bond}\right)\zeta - \frac{1}{\bond}\zeta'' & = \Omega\psi && \ \ \tr{ on $(-1, 1)$}, \\ 
\label{eq:SloshCS3} \left(h\psi'\right)(\pm 1) = \zeta(\pm 1) & = 0. 
\end{alignat} 
\end{subequations} 
For $\alpha = 0$, we must impose the necessary condition $\int_{-1}^1 \zeta\, dx = 0$ as well. Note that the governing equations \eqref{eq:SloshCS1} and \eqref{eq:SloshCS2} can also be derived from averaging \eqref{eq:SloshCanal} from $z = -h(x)$ to $z = 0$ and collecting $\O(h)$ terms. With this derivation, $\psi(x)$ is identified as $\varphi(x, 0)$ instead; see \cite[Section 10.13]{Stoker}. 


We now state the \emph{isoperimetric sloshing problem for shallow canals:} For every $\alpha\ge 0$, find the shape function $h\in\M_A$ that maximizes $\Omega_{\alpha, 1}$, \ie solve 
\begin{equation} \label{eq:Iso2D} 
 \sup_{h\in\M_A} \Omega_{\alpha, 1}(h). 
\end{equation} 
Our approach in solving \eqref{eq:Iso2D} is based on the following simple observation. Let $(\psi_\alpha^*, \zeta_\alpha^*)$ be an admissible pair of trial functions in \eqref{eq:Slosh2D_VCs} satisfying 
\begin{equation}\label{eq:Iso2D_opt} 
\left(\left(\psi_\alpha^*\right)'\right)^2 + \alpha^2(\psi_\alpha^*)^2 = c_\alpha^2 \ \ \tr{ on $(-1, 1)$} 
\end{equation} 
for some constant $c_\alpha\neq 0$. Then we obtain a simple upper bound for $\Omega_{\alpha, 1}(h)$: 
\begin{equation}\label{eq:Iso2D_opt1} 
\Omega_{\alpha, 1}(h)\le \frac{1}{2}\int_{-1}^1 hc_\alpha^2\, dx + S_{\C, \alpha}(\zeta_\alpha^*) = \frac{c_\alpha^2A}{2} + S_{\C, \alpha}(\zeta_\alpha^*) \ \ \tr{ for any $h\in\M_A$}. 
\end{equation} 
Moreover, equality holds in \eqref{eq:Iso2D_opt1} only if $(\psi_\alpha^*, \zeta_\alpha^*)$ is a fundamental shallow sloshing eigenfunction of \eqref{eq:SloshCanalS} associated with some admissible shape function $h_\alpha^*\in\M_A$. This suggests the following strategy in solving \eqref{eq:Iso2D}: 
\begin{enumerate} 
\item Solve \eqref{eq:Iso2D_opt} for $\psi_\alpha^*$ and substitute $\psi_\alpha^*$ into \eqref{eq:SloshCS2} and \eqref{eq:SloshCS3} to solve for $\zeta_\alpha^*$. 
\item Substitute $(\psi_\alpha^*, \zeta_\alpha^*)$ from Step 1 into \eqref{eq:SloshCS1} and \eqref{eq:SloshCS3} to solve for $h_\alpha^*$. 
\item Compute $\Omega_{\alpha, 1}^*\coloneqq \Omega_{\alpha, 1}(h_\alpha^*)$ using $\int_{-1}^1 h_\alpha^*\, dx = A$. Then check that $h_\alpha^*\in\M_A$. 
\item Check that $(\Omega_{\alpha, 1}^*, \psi_\alpha^*, \zeta_\alpha^*)$ is a fundamental shallow sloshing eigenpair of \eqref{eq:SloshCanalS}. This is numerically verified using the finite difference method with a standard second-order central difference. 
\end{enumerate} 
Note that from \eqref{eq:Iso2D_opt}, $\psi_\alpha^*$ is proportional to $c_\alpha$ and it follows from \eqref{eq:SloshCS2} that $\zeta_\alpha^*$ is also proportional to $c_\alpha$. Since we are only interested in $\Omega_{\alpha, 1}^*$ and \eqref{eq:Slosh2D_VCs} is equivalent to minimizing $E_\alpha/|G|$ over all nonzero functions $(\psi, \zeta)\in\X_\alpha$, we may choose $c_\alpha$ to be any positive real number in \eqref{eq:Iso2D_opt}, \ie the amplitude of $\psi_\alpha^*$ is irrelevant. We summarize these observations in the following theorem. 

\begin{theorem} \label{thm:Iso2D_opt} 
Given $\alpha\ge 0$, if there exists $h_\alpha^*\in\M_A$ such that $(\Omega_{\alpha, 1}(h_\alpha^*), \psi_\alpha^*, \zeta_\alpha^*)\in\R\times\X_\alpha$ is a fundamental weak shallow sloshing eigenpair with $\psi_\alpha^*$ satisfying \eqref{eq:Iso2D_opt} for some $c_\alpha^2\neq 0$, then $h_\alpha^*$ is a solution of \eqref{eq:Iso2D}, \ie $h_\alpha^*$ is a maximizer of $\Omega_{\alpha, 1}$. 
\end{theorem} 

\begin{remark} \label{rm:Iso2D_opt} 
Formally, \eqref{eq:Iso2D_opt} can be interpreted as the first-order optimality condition for \eqref{eq:Iso2D}. Suppose $\Omega_{\alpha, 1}(h_\alpha^*)$ is differentiable. For sufficiently small $\vareps > 0$, consider the family of functions $h_\alpha(x; \vareps) = h_\alpha^*(x) + \vareps v(x)$ for any piecewise smooth variation $v(x)$ satisfying $h_\alpha(x; \vareps)\ge 0$ on $[-1, 1]$ and $\int_{-1}^1 h_\alpha(x; \vareps)\, dx = A$. Differentiating the area constraint with respect to $\vareps$ yields $\int_{-1}^1 v\,dx = 0$. For every $\vareps > 0$, let $(\Omega_{\alpha, 1}(\vareps), \psi_\alpha(x; \vareps), \zeta_\alpha(x; \vareps))$ denote a fundamental weak shallow sloshing eigenpair associated with $h_\alpha(x; \vareps)$ and define $\od{\Omega_\alpha}\coloneqq \frac{d\Omega_{\alpha, 1}}{d\vareps}$. Differentiating the weak formulation with respect to $\vareps$, we see that the following holds for all $(f, g)\in\X_\alpha$: 
\begin{align} 
& \int_{-1}^1 \left[h_\alpha\left(\od{\psi_\alpha}'f' + \alpha^2\od{\psi_\alpha}f\right) + \left(1 + \frac{\alpha^2}{\bond}\right)\od{\zeta_\alpha}g + \frac{1}{\bond}\od{\zeta_\alpha}'g'\right] dx \label{eq:Iso2D_opt2} \\ 
& + \int_{-1}^1 v\left(\psi_\alpha'f' + \alpha^2\psi_\alpha f\right) dx = \Omega_\alpha\int_{-1}^1 \left(\od{\zeta_\alpha}f + \od{\psi_\alpha}g\right) dx + \od{\Omega_\alpha}\int_{-1}^1 (\zeta_\alpha f + \psi_\alpha g) dx. \notag 
\end{align} 
Choosing $(f, g) = \left(\psi_\alpha(x; 0), \zeta_\alpha(x; 0)\right)$ in \eqref{eq:Iso2D_opt2} and $(f, g) = (\od{\psi_\alpha}, \od{\zeta_\alpha})$ in Definition \ref{def:Slosh2D_weak_s} corresponding to $\left(\Omega_{\alpha, 1}(0), \psi_\alpha(x; 0), \zeta_\alpha(x; 0)\right)$, setting $\vareps = 0$, and using the assumption that $\od{\Omega_\alpha}(0) = 0$, we are left with 
\begin{equation} \label{eq:Iso2D_opt3} 
\int_{-1}^1 v\left[(\psi_\alpha'(x; 0))^2 + \alpha^2\psi_\alpha^2(x; 0)\right] dx = 0. 
\end{equation} 
This yields \eqref{eq:Iso2D_opt} since \eqref{eq:Iso2D_opt3} must hold for all $v$ satisfying $\int_{-1}^1 v\, dx = 0$. 
\end{remark} 

Theorems \ref{thm:Iso2D_zero}-\ref{thm:Iso2D_1} tell us that in the absence or presence of surface tension, the maximizing cross-section $h_\alpha^*$ for every $\alpha\ge 0$ is symmetric. We now show that this can be interpreted as a consequence of the concavity of the map $h\mapsto\Omega_{\alpha, 1}(h)$.  

\begin{theorem}\label{thm:Iso2D_sym} 
Given $\alpha\ge 0$, the map $h\mapsto\Omega_{\alpha, 1}(h)$ is concave on $\M_A$. As a consequence, if there exists a maximizer of $\Omega_{\alpha, 1}$, then there exists a symmetric maximizer of $\Omega_{\alpha, 1}$ too. 
\end{theorem} 
\begin{proof} 
Fix $\alpha\ge 0$. The concavity follows from the variational characterization of $\Omega_{\alpha, 1}$ (see \eqref{eq:Slosh2D_VCs}), as $\M_A$ is convex and the map $h\mapsto\Omega_{\alpha, 1}(h)$ is the infimum of the family of affine functions $\{h\mapsto (E_\alpha/|G|)(h; \psi, \zeta)\}_{(\psi, \zeta)}$. Next, suppose $h_1\in\M_A$ is a maximizer of $\Omega_{\alpha, 1}$. Define $h_2(x)\coloneqq h_1(-x)\in\M_A$ and $h_3\coloneqq (h_1 + h_2)/2\in\M_A$, the latter which is symmetric. It is clear from \eqref{eq:Slosh2D_VCs} that $\Omega_{\alpha, 1}(h_2) = \Omega_{\alpha, 1}(h_1)$. By concavity of the map $h\mapsto\Omega_{\alpha, 1}(h)$, we get 
\begin{align*}
\Omega_{\alpha, 1}(h_3) \ge \frac{1}{2}\Omega_{\alpha, 1}(h_1) + \frac{1}{2}\Omega_{\alpha, 1}(h_2) = \Omega_{\alpha, 1}(h_1). 
\end{align*} 
\end{proof}


\subsection{Zero surface tension: proof of Theorem \ref{thm:Iso2D_zero}} \label{sec:Iso2D_zero} 
In the absence of surface tension, the corresponding shallow sloshing problem is obtained from \eqref{eq:SloshCanalS} by formally setting $\bond = \infty$. Decoupling the equations, we see that $(\Omega_\alpha^2, \psi_\alpha)\in\R\times H^1(-1, 1)$ satisfies the following Sturm-Liouville problem: 
\begin{equation}\label{eq:SloshCanalS_zero} 
\begin{alignedat}{1} 
-(h\psi_\alpha')' + \alpha^2h\psi_\alpha & = \Omega_\alpha^2\psi_\alpha \coloneqq \lambda_\alpha^\infty\psi_\alpha \ \ \tr{ on $(-1, 1)$}, \\ 
(h\psi_\alpha')(\pm 1) & = 0. 
\end{alignedat} 
\end{equation} 
It is not difficult to see that for every $\alpha\ge 0$, the squared fundamental sloshing frequency $\lambda_{\alpha, 1}^\infty$ admits the following variational characterization: 
\begin{equation} \label{eq:Slosh2D_zero_VCs} 
\lambda_{\alpha, 1}^\infty(h) = \inf_{\psi\in H^1(-1, 1)}\left\{2D_\alpha(h; \psi)\colon \int_{-1}^1 \psi^2\, dx = 1\right\}. 
\end{equation} 
Troesch proved that the maximizing cross-section for $\alpha = 0$ is a parabola with squared maximal sloshing frequency $\lambda_{0, 1}^{\infty, *}\coloneqq 3A/2$ \cite{troesch:1965}. We now show that the maximizing cross-section for $\alpha > 0$ is a rectangle $h_\alpha^{\infty, *} = A/2$. 

\begin{proof}[\tb{Proof of Theorem \ref{thm:Iso2D_zero}}]
For any $\alpha > 0$, choosing $\psi_\alpha = 1/\sqrt{2}$ as an admissible trial function in \eqref{eq:Slosh2D_zero_VCs} yields 
\begin{align} \label{eq:Iso2D_zero1} 
\lambda_{\alpha, 1}^\infty(h) \le \int_{-1}^1 \frac{h\alpha^2}{2}\, dx = \frac{\alpha^2A}{2} \ \ \textrm{ for any $h\in\M_A$.} 
\end{align} 
To see that equality holds in \eqref{eq:Iso2D_zero1} for $h = h_\alpha^{\infty, *} = A/2$, we substitute $h = A/2$ into \eqref{eq:SloshCanalS_zero} and rearrange to obtain 
\begin{align*} 
\psi_\alpha'' + \left(\frac{2\lambda_\alpha^\infty}{A} - \alpha^2\right)\psi_\alpha = 0 \ \ \tr{ on $(-1, 1)$}; \ \ \psi_\alpha'(\pm 1) = 0. 
\end{align*} 
It is clear that $(\lambda_\alpha^\infty, \psi_\alpha) = (\alpha^2A/2, c)$ is the fundamental eigenpair for any nonzero constant $c$. 
\end{proof}


\subsection{Finite surface tension: proof of Theorems \ref{thm:Iso2D_0} and \ref{thm:Iso2D_1}} \label{sec:Iso2D_OC} 
Throughout this subsection, for any given $\alpha > 0$ we denote by $h$ the maximizing cross-section, $(\Omega, \psi, \zeta)$ its corresponding fundamental shallow sloshing eigenpair, and $\lambda\coloneqq \Omega^2$ for notational convenience. 

\begin{proof}[\tb{Proof of Theorem \ref{thm:Iso2D_0}}] 
Set $\alpha = 0$. Choosing $c_0 = 1$ in \eqref{eq:Iso2D_opt}, one such solution is $\psi(x) = x + d_0$ for some $d_0\in\R$. We first solve for $\zeta$. Define $\kappa\coloneqq \sqrt{\bond}$. Substituting $\psi(x) = x + d_0$ into \eqref{eq:SloshCS2} for $\alpha = 0$ and rearranging yield 
\begin{equation*} 
\zeta'' - \kappa^2\zeta = -\Omega\kappa^2\left(x + d_0\right) \ \ \tr{ on $(-1, 1)$}.   
\end{equation*} 
Together with $\zeta(\pm 1) = 0$ and $\int_{-1}^1 \zeta\, dx = 0$, the solution is given by
\[ \zeta(x) = \Omega \left(x - \frac{\sinh(\kappa x)}{\sinh\kappa}\right). \] 
Next we solve for $h$. Substituting $\psi$ and $\zeta$ into \eqref{eq:SloshCS1} for $\alpha = 0$ and \eqref{eq:SloshCS3} yield 
\[ h' = -\lambda\left[x - \frac{\sinh(\kappa x)}{\sinh\kappa}\right]; \ \ h(\pm 1) = 0. \] 
The solution is given by
\[ h(x) = \frac{\lambda}{2}(1 - x^2) - \frac{\lambda}{\kappa\sinh\kappa}\Big(\cosh\kappa - \cosh(\kappa x)\Big). \] 
A direct computation of $\int_{-1}^1 h\, dx = A$ shows that $\lambda = \lambda_{0, 1}^* > 0$ as defined in \eqref{eq:Iso2D_0_eigs}. Finally, we have $h\in\M_A$ as $h$ is even, $h(1) = 0$, and $h' < 0$ on $(0, 1)$; the latter follows from the fact that $\sinh(z)/z$ is strictly increasing for $z > 0$. 
\end{proof} 

\begin{proof}[\tb{Proof of Theorem \ref{thm:Iso2D_1}}] 
Fix $\alpha > 0$. Choosing $c_\alpha = \alpha^2$ in \eqref{eq:Iso2D_opt}, one such solution is $\psi = 1$. We first solve for $\zeta$. Define $\kappa_\alpha\coloneqq \sqrt{\alpha^2 + \bond}$. Substituting $\psi = 1$ into \eqref{eq:SloshCS2} and rearranging yield 
\begin{equation*} 
\zeta'' - \kappa_\alpha^2\zeta = -\Omega\bond \ \ \tr{ on $(-1, 1)$}. 
\end{equation*} 
Together with $\zeta(\pm 1) = 0$, the solution is given by
\[ \zeta(x) = \frac{\Omega\bond}{\kappa_\alpha^2}\left[1 - \frac{\cosh(\kappa_\alpha x)}{\cosh\kappa_\alpha}\right]. \] 
Next we solve for $h$. Since $(h\psi')(\pm 1) = 0$ is trivially satisfied with $\psi = 1$, it follows from \eqref{eq:SloshCS1} that 
\[ h(x) = \frac{\Omega\zeta(x)}{\alpha^2\psi(x)} = \frac{\lambda\bond}{\alpha^2\kappa_\alpha^2}\left[1 - \frac{\cosh(\kappa_\alpha x)}{\cosh\kappa_\alpha}\right]. \] 
Again, a direct computation of $\int_{-1}^1 h\, dx = A$ shows that $\lambda = \lambda_{\alpha, 1}^* > 0$ as defined in \eqref{eq:Iso2D_1_eigs}. Finally, it is clear that $h\in\M_A$ as $h$ is even, $h(1) = 0$, and $h$ is strictly decreasing on $(0, 1)$. 
\end{proof} 

\begin{remark} 
Given $\alpha\ge 0$, let $h_1$ and $h_2$ be two maximizing cross-section. By concavity of $h\mapsto\Omega_{\alpha, 1}(h)$ (see Theorem \ref{thm:Iso2D_sym}), we know that the average $h_3\coloneqq (h_1 + h_2)/2$ is also a maximizing cross-section. Let $(\psi_\alpha, \zeta_\alpha)$ be an eigenfunction associated with $\Omega_{\alpha, 1}(h_3)$. From the variational characterization \eqref{eq:Slosh2D_VCs} together with linearity of the map $h\mapsto E_\alpha(h; \psi, \zeta)$, we get 
\begin{align*}
\Omega_{\alpha, 1}(h_3) = E_\alpha(h_3; \psi_\alpha, \zeta_\alpha) & = \frac{1}{2}E_\alpha(h_1; \psi_\alpha, \zeta_\alpha) + \frac{1}{2}E(h_2; \psi, \zeta) \\ 
& \ge \frac{1}{2}\Omega_{\alpha, 1}(h_1) + \frac{1}{2}\Omega_{\alpha, 1}(h_2). 
\end{align*} 
By extremality of $h_1$ and $h_2$, it must be the case that $(\psi_\alpha, \zeta_\alpha)$ is an eigenfunction associated with both $\Omega_{\alpha, 1}(h_1)$ and $\Omega_{\alpha, 1}(h_2)$. Consequently, we have 
\begin{equation} \label{eq:Iso2D_uq1} 
\int_{-1}^1 \left(h_1 - h_2\right)\Big[\psi_\alpha'f' + \alpha^2\psi_\alpha f\Big]\, dx = 0 \ \ \tr{ for all $f\in H^1(-1, 1)$.} 
\end{equation} 
We claim that \emph{the maximizing cross-section is unique under the assumption that the eigenfunction associated with any maximizer of \eqref{eq:Iso2D} is unique.} For $\alpha = 0$, \eqref{eq:Iso2D_uq1} gives $(h_1 - h_2)\psi_0' = C$ for some constant $C$. Since $\psi_0'\neq 0$ from the proof of Theorem \ref{thm:Iso2D_0}, we must have $h_1 = h_2$ on $(-1, 1)$. For $\alpha > 0$, \eqref{eq:Iso2D_uq1} gives $\left[(h_1 - h_2)\psi_\alpha'\right]' = (h_1 - h_2)\alpha^2\psi_\alpha$. Since $\psi_\alpha$ is constant from the proof of Theorem \ref{thm:Iso2D_1}, we again have $h_1 = h_2$ on $(-1, 1)$. 
\end{remark} 

\begin{figure}[t]  
\begin{center}
	\includegraphics[width = 0.49\textwidth]{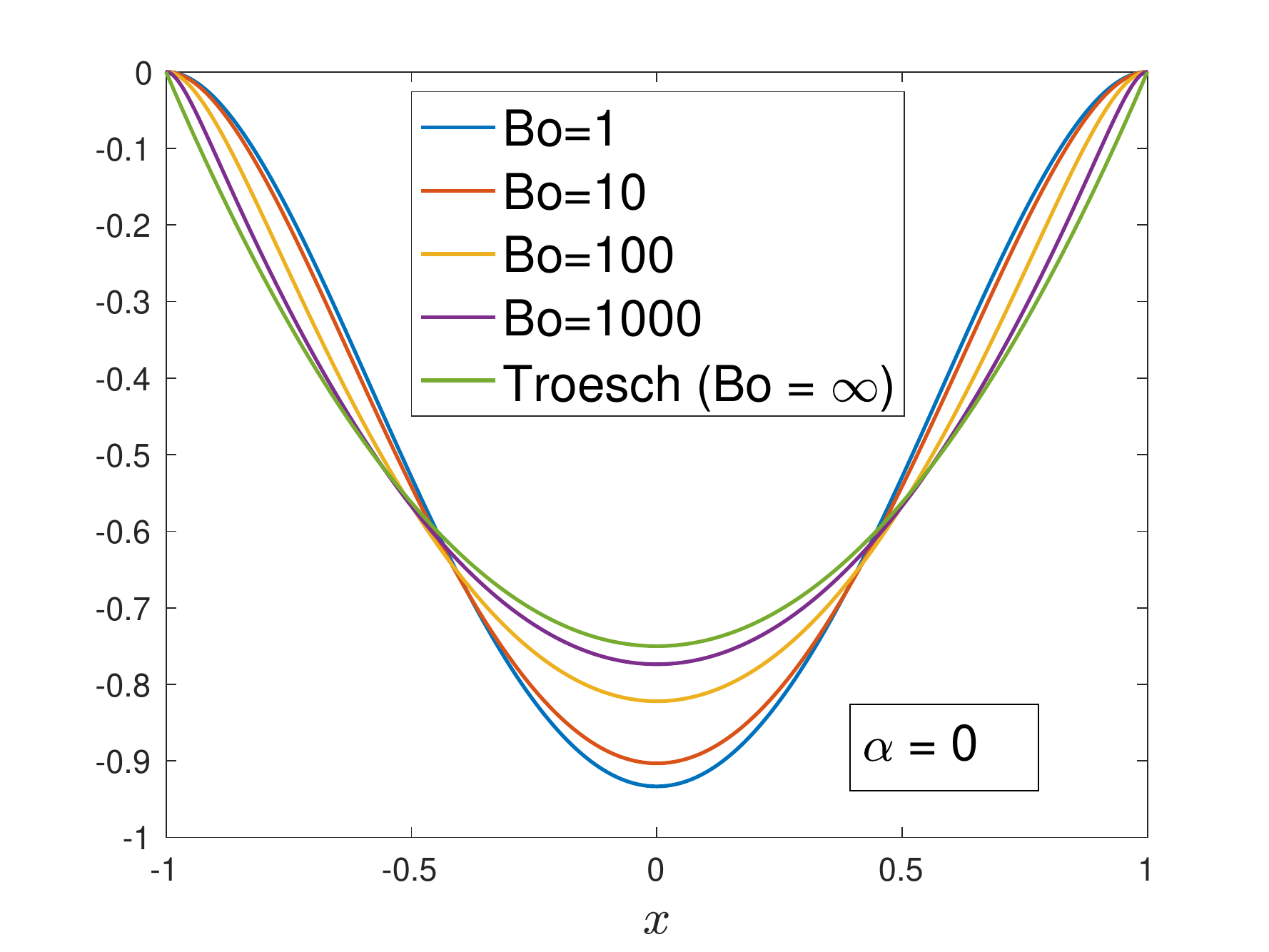} 
	\includegraphics[width = 0.49\textwidth]{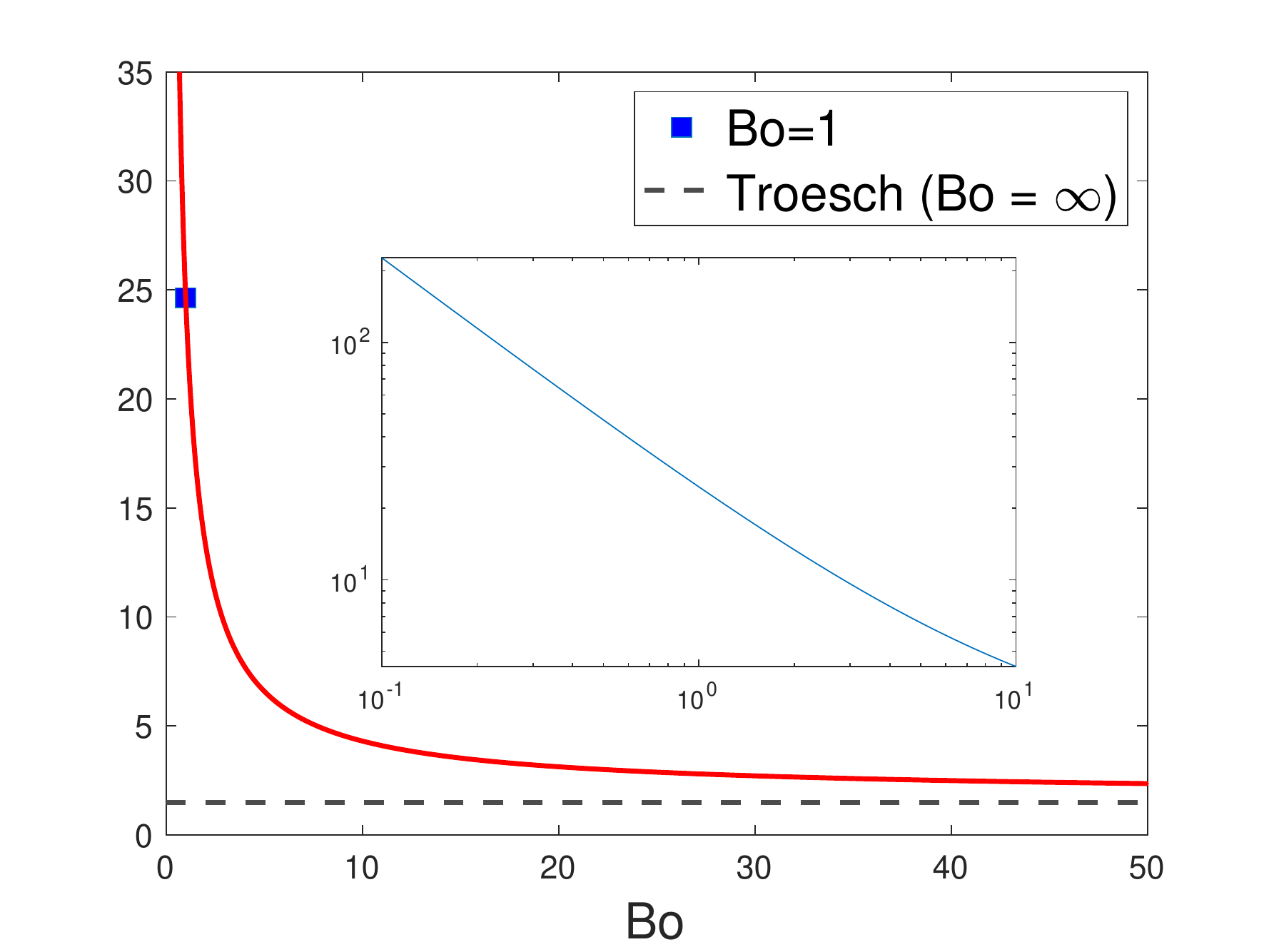} 
	\includegraphics[width = 0.49\textwidth]{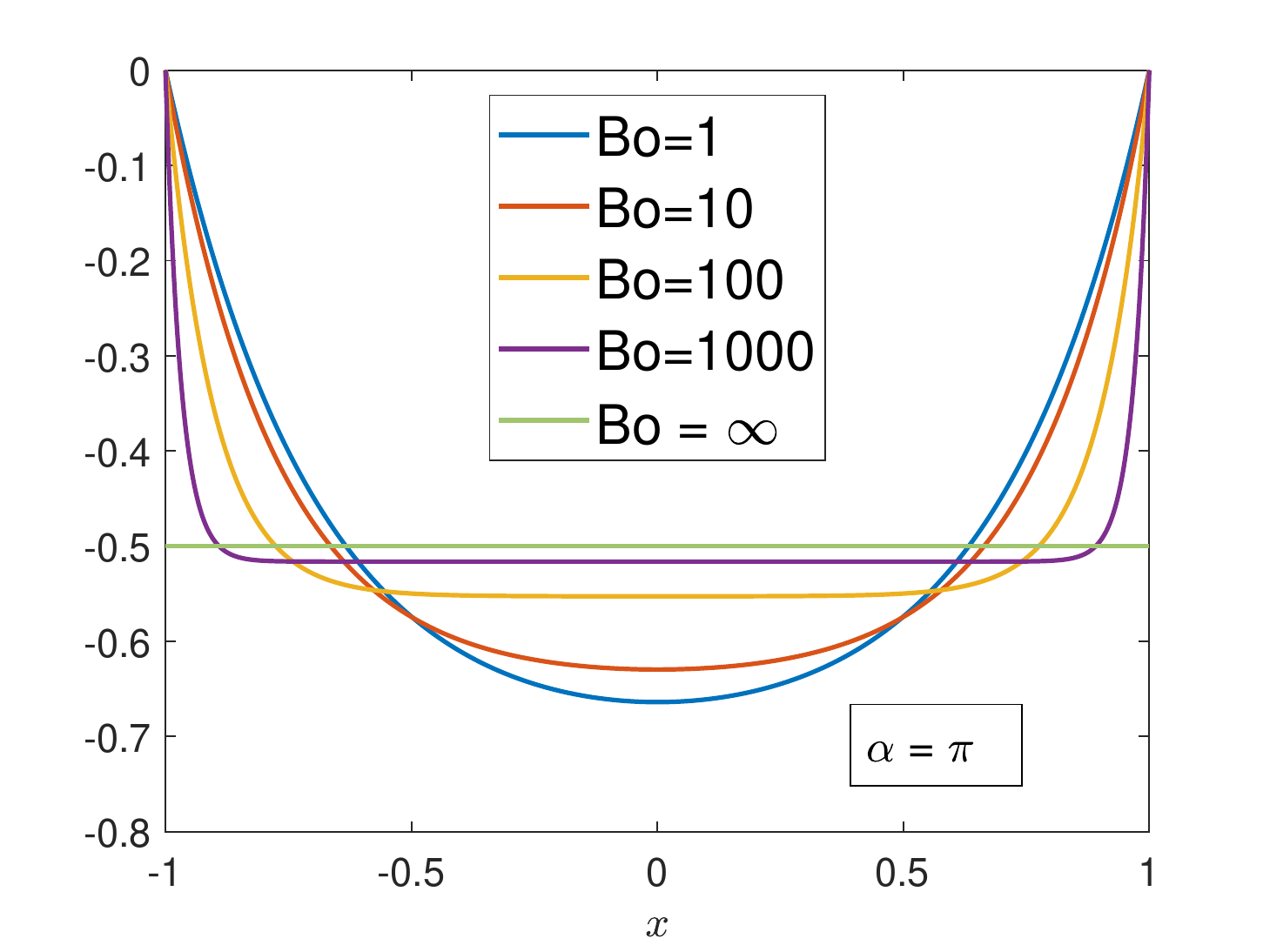} 
	\includegraphics[width = 0.49\textwidth]{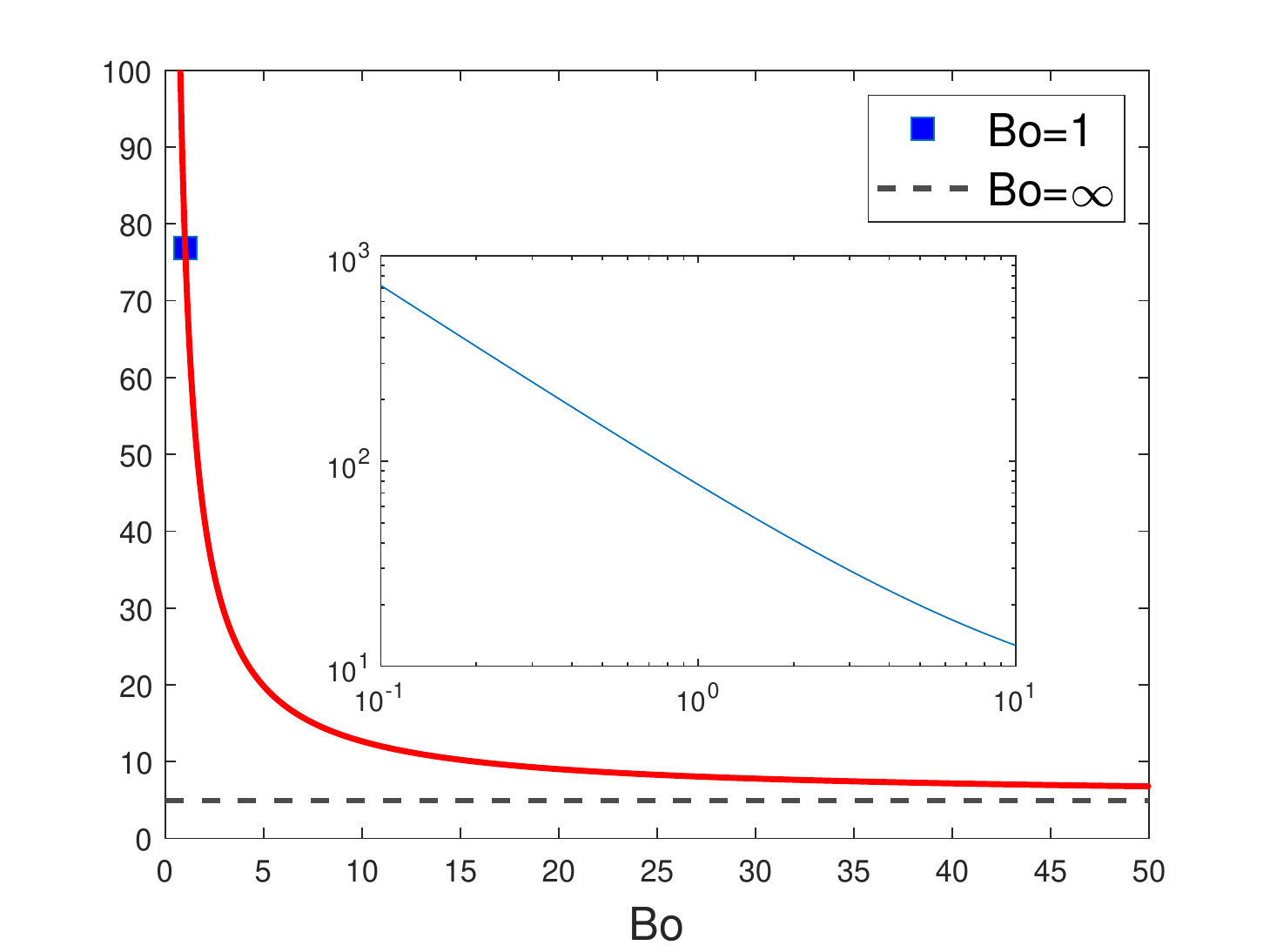} 
\caption{An illustration of the zero surface tension limit for $\alpha = 0$ \tb{(top)} and $\alpha = \pi$ \tb{(bottom)}, both with $A = 1$. \tb{(Left)} The maximizing cross-section plotted for varying $\bond$ (see \eqref{eq:Iso2D_0_shape} and \eqref{eq:Iso2D_1_shape}) and for $\bond = \infty$ (zero surface tension). \tb{(Right)} The squared maximal sloshing frequency \eqref{eq:Iso2D_0_eigs} and \eqref{eq:Iso2D_1_eigs} plotted as a function of $\bond$. The inset is a log-log plot of the squared maximal sloshing frequency for $\bond\in [0.1, 10]$.} 
\label{fig:ST_alpha} 
\end{center}
\end{figure}


\subsection{Zero surface tension limit $(\bond\to\infty)$} \label{sec:Iso2D_zeroST} 
In this subsection, we show that for every $\alpha\ge 0$, the corresponding optimal shallow canal without surface tension is the zero surface tension limit of the optimal shallow canal with surface tension. Moreover, the map $\bond\mapsto\lambda_{\alpha, 1}^*(\bond)$ is strictly decreasing on $(0, \infty)$. Figure \ref{fig:ST_alpha} illustrates these results for $\alpha = 0$ (top) and $\alpha = \pi$ (bottom). For $\bond = 1$, we get $\lambda_{0, 1}^*/\lambda_{0, 1}^{\infty, *}$ and $\lambda_{\pi, 1}^*/\lambda_{\pi, 1}^{\infty, *}$ to be approximately 16.4 and 15.6, \ie \emph{the squared maximal sloshing frequency increases drastically when capillary and gravitational forces are comparable}. In Figure \ref{fig:ST_alpha}(right), the log-log plots reveal that $\lambda_{0, 1}^*(\bond)\propto\bond^{-0.808}$ and $\lambda_{\pi, 1}^*(\bond)\propto\bond^{-0.8307}$ for $\bond\in [0.1, 10]$.

\begin{corollary}[$\alpha = 0, \bond\to\infty$] \label{thm:Iso2D_zeroST_0} 
Let $\lambda_{0, 1}^*$, $h_0^*$ be defined as in Theorem \ref{thm:Iso2D_0}. The map $\bond\mapsto\lambda_{0, 1}^*(\bond)$ is strictly decreasing on $(0, \infty)$ and $\lambda_{0, 1}^*(\bond)\to \lambda_{0, 1}^{\infty, *} = 3A/2$ as $\bond\to\infty$. Moreover, $h_0^*(\bond; x)\to h_0^{\infty, *}(x) \coloneqq 3A(1 - x^2)/4$ as $\bond\to\infty$ for every $x\in [-1, 1]$ . 
\end{corollary} 
\begin{proof} 
Define $\kappa\coloneqq \sqrt{\bond} > 0$. From \eqref{eq:Iso2D_0_eigs}, we write $\lambda_{0, 1}^*$ as  
\[ \lambda_{0, 1}^*(\bond) = \lambda_{0, 1}^*(\kappa^2) = \lambda_{0, 1}^{\infty, *}\Big(1 - 3Y(\kappa)\Big)^{-1},  \tr{ with } Y(\kappa) = \frac{\kappa - \tanh\kappa}{\kappa^2\tanh\kappa}. \] 
To prove the first statement, we need to show that $Y$ is strictly decreasing on $(0, \infty)$, with range $(0, 1/3)$. This was proven in \cite[Example 4(3), p. 809]{anderson:2006monotonicity}. Lastly, we compute the zero surface tension limit of $h_0^*(\bond; x)$. Comparing \eqref{eq:Iso2D_0_shape} to $h_0^{\infty, *}(x)$, we need only show the second term in \eqref{eq:Iso2D_0_shape} vanishes in the limit of $\bond\to\infty$ for any $x\in [-1, 1]$. This is evident for $x = \pm 1$. For any fixed $x\in (-1, 1)$, this follows from the fact that $\cosh(\kappa x)/\cosh\kappa\to 0$ as $\kappa\to\infty$. 
\end{proof} 

\begin{corollary}[$\alpha > 0$, $\bond\to\infty$] \label{thm:Iso2D_zeroST_1}
For any $\alpha > 0$, let $\lambda_{\alpha, 1}^*$, $h_\alpha^*$ and $\lambda_{\alpha, 1}^{\infty, *}$, $h_\alpha^{\infty, *}$ be defined as in Theorems \ref{thm:Iso2D_1} and \ref{thm:Iso2D_zero}, respectively. The map $\bond\mapsto\lambda_{\alpha, 1}^*(\bond)$ is strictly decreasing on $(0, \infty)$ and $\lambda_{\alpha, 1}^*(\bond)\to \lambda_{\alpha, 1}^{\infty, *}$ as $\bond\to\infty$. Moreover, $h_\alpha^*(\bond; x)\to h_\alpha^{\infty, *}(x)$ as $\bond\to\infty$ for every $x\in (-1, 1)$. 
\end{corollary} 
\begin{proof} 
Fix $\alpha > 0$. From \eqref{eq:Iso2D_1_eigs}, we write $\lambda_{\alpha, 1}^*(\bond) = \lambda_{\alpha, 1}^{\infty, *}Z(\bond)\left[1 - Y(\bond)\right]^{-1}$, where 
\[ Z(\bond) = \frac{\kappa_\alpha^2}{\bond} = 1 + \frac{\alpha^2}{\bond}, \ \ Y(\bond) = \frac{\tanh\kappa_\alpha}{\kappa_\alpha} = \frac{\tanh\left(\sqrt{\alpha^2 + \bond}\right)}{\sqrt{\alpha^2 + \bond}}. \] 
The first statement now follows immediately. For any fixed $x\in (-1, 1)$, the zero surface tension limit of $h_\alpha^*(\bond; x)$ follows from the fact that $\cosh(\kappa_\alpha x)/\cosh\kappa_\alpha\to 0$ as $\bond\to\infty$. 
\end{proof}


\section{Radially symmetric containers} \label{sec:RS} 
We now assume that $\D$ is radially symmetric, \ie $\D\setminus\Gamma_0$ is generated by the rotation of a planar meridian domain $\RR$ about the $z$-axis, where $\Gamma_0$ is the part of the $z$-axis contained in $\D$; see Figure \ref{fig:container}(right). We choose the radius $r_0$ of the equilibrium free surface as the characteristic length scale. Let $(r, \theta, z)$ be dimensionless cylindrical coordinates with $\theta\in T\coloneqq (-\pi, \pi]$. Then $\D\setminus\Gamma_0$ is transformed to $\RR\times T$ and we write $\RR = \left\{(r, z)\colon r\in (0, 1), \, z\in (-h(r), 0)\right\}$ with boundary $\del\RR = \overline{\Gamma_0}\cup\B_\RR\cup\F_\RR\cup\{(1, 0)\}$. Here, $\B_\RR$ is the graph of $z = -h(r)$ with $0 < r < 1$ and $\F_\RR = (0, 1)\times \{0\}$. 

In cylindrical coordinates, the functions $\Phi$ and $\xi$ defined on $\D$ and $\F$, respectively, can be represented by the Fourier series $\sum_{m\in\Z} \varphi_m(r, z)e^{im\theta}$ and $\sum_{m\in\Z} \zeta_m(r)e^{im\theta}$, respectively. Since we are interested in real-valued solutions, we consider the corresponding real Fourier series. It is clear that $\Phi = \varphi_m(r, z)\Theta(m\theta)$ and $\xi = \zeta_m(r)\Theta(m\theta)$, with $m = 0, 1, 2, \dots$ and $\Theta(m\theta)$ being $\cos(m\theta)$ or $\sin(m\theta)$, reduces \eqref{eq:SloshST} in cylindrical coordinates to the following infinite sequence of two-dimensional generalized mixed Steklov problems for $(\omega_m, \varphi_m, \zeta_m)$: 
\begin{subequations} \label{eq:SloshRS} 
\begin{alignat}{3} 
\label{eq:SloshRS1} r^{-1}\del_r(r\del_r\varphi_m) + \del_{zz}\varphi_m & = \frac{m^2}{r^2}\varphi_m && \ \ \tr{ in } \ \ && \RR, \\ 
\label{eq:SloshRS2} \del_\n\varphi_m & = 0 && \ \ \tr{ on } && \B_{\RR}, \\ 
\label{eq:SloshRS3} \del_z\varphi_m & = \omega_m\zeta && \ \ \tr{ on } && \F_{\RR}, \\ 
\label{eq:SloshRS4} \zeta_m - \frac{1}{\bond}\left(r^{-1}(r\zeta_m')' - \frac{m^2}{r^2}\zeta_m\right)& = \omega_m\varphi_m && \ \ \tr{ on } \ \ && \F_\RR, \\ 
\label{eq:SloshRS5} \zeta_m(1) & = 0, 
\end{alignat} 
\end{subequations} 
where we now have one contact point $(r, z) = (1, 0)$ in \eqref{eq:SloshRS}. It is straightforward to verify that the ansatz for $\xi$ satisfies the necessary condition $\int_\F\xi\, dA = 0$ due to the factor $\Theta(m\theta)$ for $m > 0$. The case $m = 0$ corresponds to a purely radial motion and a necessary condition for the existence of solution is $\int_0^1 \zeta_0 r\, dr = 0$.


\subsection{Variational principle} \label{sec:RS_variational} 
We first define suitable function spaces for the Fourier coefficients $\varphi_m$ defined on $\RR$. For any $k\in \R$, we consider the weighted $L^2$-spaces on $\RR$ with weight $r^k$ 
\begin{align*} 
L_k^2(\RR)\coloneqq \left\{\varphi\colon\RR\to\R\colon \|\varphi\|_{L_k^2(\RR)}^2\coloneqq \int_\RR \varphi^2r^k\, drdz < \infty\right\}. 
\end{align*} 
For $m = 0$, we consider the following weighted Sobolev-space on $\RR$ with weight $r$ 
\begin{align*} 
H_1^1(\RR) & \coloneqq \left\{\varphi\in L_1^2(\RR)\colon \del_r\varphi, \del_z\varphi\in L_1^2(\RR)\right\}, 
\end{align*} 
with norm $\|\varphi\|_{H_1^1(\RR)}^2 \coloneqq \|\varphi\|_{L_1^2(\RR)}^2 + \|\del_r\varphi\|_{L_1^2(\RR)}^2 + \|\del_z\varphi\|_{L_1^2(\RR)}^2$. For $m > 0$, the factor $m^2/r^2$ in \eqref{eq:SloshRS1} prompts us to considering a separate weighted Sobolev space 
\[ V_1^1(\RR)\coloneqq H_1^1(\RR)\cap L_{-1}^2(\RR), \] 
with norm $\|\varphi\|_{V_1^1(\RR)}^2 \coloneqq \|\varphi\|_{L_{-1}^2(\RR)}^2 + \|\del_r\varphi\|_{L_1^2(\RR)}^2 + \|\del_z\varphi\|_{L_1^2(\RR)}^2$. The function spaces $L_k^2(0, 1)$, $H_1^1(0, 1)$, and $V_1^1(0, 1)$ can be defined analogously by replacing the measure $drdz$ with $dr$. Since $\zeta_m$ satisfies the Dirichlet boundary condition $\zeta_m(1) = 0$, we introduce the subspaces $\os{H_1^1}(0, 1)$ and $\os{V_1^1}(0, 1)$ of functions that vanish at $r = 1$. 

For $m = 1, 2, 3, \dots$, we introduce the Hilbert space $\H_\RR\coloneqq V_1^1(\RR)\times \os{V_1^1}(0, 1)$ with norm $\|(\varphi, \zeta)\|_{\H_\RR}^2\allowbreak\coloneqq \|\varphi\|_{V_1^1(\RR)}^2 + \|\zeta\|_{V_1^1(0, 1)}^2$. Suppose $(\omega_m, \varphi_m, \zeta_m)$ is a sufficiently regular solution of \eqref{eq:SloshRS}. Testing \eqref{eq:SloshRS1}, \eqref{eq:SloshRS4} with $fr$ and $gr$ respectively, with $(f, g)\in\H_\RR$, and using the remaining equations in \eqref{eq:SloshRS}, we arrive at the following weak formulation of \eqref{eq:SloshRS} for $m = 1, 2, 3, \dots$. To this end, define $\nabla\coloneqq (\del_r, \del_z)$. 

\begin{definition} 
Given $m = 1, 2, 3, \dots$, we say that $(\omega_m, \varphi_m, \zeta_m)\in\R\times\H_\RR$, $(\varphi_m, \zeta_m)\neq (\0, \0)$ is a weak sloshing eigenpair of \eqref{eq:SloshRS} if the following holds for all $(f, g)\in\H_\RR$: 
\begin{align*} 
\int_\RR \left(\nabla\varphi_m\cdot\nabla f + \frac{m^2}{r^2}\varphi_mf\right)r\, drdz & + \int_0^1 \left[\zeta_mg + \frac{1}{\bond}\left(\zeta_m'g' + \frac{m^2}{r^2}\zeta_mg\right)\right]r\, dr \\ 
& = \omega_m\int_0^1 \left(\zeta_mf + \varphi_mg\right)r\, dr. 
\end{align*} 
\end{definition} 

For $m = 0$, it is necessary to introduce the Hilbert space defined by 
\[ \H_{\RR, 0}\coloneqq \left\{(\varphi, \zeta)\in H_1^1(\RR)\times \os{H_1^1}(0, 1)\colon \int_0^1 \zeta r\, dr = 0\right\}, \]
 with norm $\|(\varphi, \zeta)\|_{\H_{\RR, 0}}^2 \coloneqq \|\varphi\|_{H_1^1(\RR)}^2 + \|\zeta\|_{H_1^1(0, 1)}^2$. 

\begin{definition} 
We say that $(\omega_0, \varphi_0, \zeta_0)\in\R\times\H_{\RR, 0}$, $(\varphi_0, \zeta_0)\neq (\0, \0)$ is a weak sloshing eigenpair of \eqref{eq:SloshRS} for $m = 0$ if the following holds for all $(f, g)\in\H_{\RR, 0}$: 
\begin{align*} 
\int_\RR \left(\nabla\varphi_0\cdot\nabla f\right) r\, drdz & + \int_0^1 \left[\zeta_0g + \frac{1}{\bond}\zeta_0'g'\right]r\, dr = \omega_0\int_0^1 \left(\zeta_0f + \varphi_0g\right)r\, dr. 
\end{align*} 
\end{definition} 

Arguing as in the case of canals, we may restrict our attention to weak sloshing eigenpairs with $\omega_m > 0$, and show that imposing $G_\RR(\varphi, \zeta) = \int_{\F_\RR} \varphi\zeta r\, dr > 0$ is a sufficient condition for obtaining positive sloshing frequencies. Similar to the energy functional $E_{\C, \alpha}$ for canals, we define the energy functional $E_{\RR, m}(\varphi, \zeta)$ given by 
\begin{equation*}
E_{\RR, m}(\varphi, \zeta)\coloneqq \underbrace{\frac{1}{2}\int_\RR \bigg[|\nabla\varphi|^2 + \frac{m^2}{r^2}\varphi^2\bigg]r\, drdz}_{D_{\RR, m}(\varphi)} + \underbrace{\frac{1}{2}\int_0^1 \bigg[\zeta^2 + \frac{1}{\bond}\bigg((\zeta')^2 + \frac{m^2}{r^2}\zeta^2\bigg)\bigg]r\, dr}_{S_{\RR, m}(\zeta)}. 
\end{equation*} 
For every $m = 0, 1, 2, \dots$, let $\omega_{m, 1}$ denote the fundamental sloshing frequency of \eqref{eq:SloshRS} with corresponding weak fundamental sloshing eigenfunction $(\varphi_{m, 1}, \zeta_{m, 1})$. We now establish a variational characterization for $\omega_{m, 1}$. 

\begin{theorem}[Variational characterization, $m\ge 1$] \label{thm:Slosh3D_VC2} 
Let $\RR$ be a bounded Lipschitz domain in $\R^2$. For every $m = 1, 2, 3, \dots$, there exists a weak fundamental sloshing eigenpair $(\omega_{m, 1}, \varphi_{m, 1}, \zeta_{m, 1})$ of \eqref{eq:SloshRS}, where $(\varphi_{m, 1}, \zeta_{m, 1})$ is a constrained minimizer of the following variational problem: 
\begin{equation} \label{eq:Slosh3D_VC2} 
\omega_{m, 1}\coloneqq \inf_{(\varphi, \zeta)\in\H_\RR} \left\{E_{\RR, m}(\varphi, \zeta)\colon G_\RR(\varphi, \zeta) = 1\right\}. 
\end{equation} 
\end{theorem} 
\begin{proof} 
Fix $m\ge 1$. Define the admissible set $M\coloneqq \left\{(\varphi, \zeta)\in\H_\RR\colon G_\RR(\varphi, \zeta) = 1\right\}$. Adapting the proof of \cite[Lemma 3.6]{tan:2017}, we see that $M$ is weakly closed in $\H_\RR$ as $L_{-1}^2(\cdot)\subset L_1^2(\cdot)$ and the embeddings $V_1^1(\RR)\hookrightarrow L_1^2(\F_\RR)$ and $V_1^1(0, 1)\hookrightarrow L_1^2(0, 1)$ are both compact \cite[Lemma 4.2]{mercier:1982}. Next we prove a coercivity estimate for $E_{\RR, m}$. Applying H\"{o}lder's and Young's inequalities, we get for any $\delta > 0$, 
\begin{align*} 
2S_{\RR, m}(\zeta) & \ge - \|\zeta^2/r\|_{L^1(0, 1)}\|r^2\|_{L^\infty(0, 1)} +  \frac{1}{\bond}\int_0^1 \left((\zeta')^2 + \frac{m^2}{r^2}\zeta^2\right)r\, dr \\ 
& \ge -\frac{\delta}{2}\|\zeta^2/r\|_{L^1(0, 1)} - \frac{1}{2\delta}\|r^2\|_{L^\infty(0, 1)} +  \frac{1}{\bond}\int_0^1 \left((\zeta')^2 + \frac{m^2}{r^2}\zeta^2\right)r\, dr \\ 
& = \left(\frac{m^2}{\bond} - \frac{\delta}{2}\right)\|\zeta\|_{L_{-1}^2(0, 1)}^2 + \frac{1}{\bond}\|\zeta'\|_{L_1^2(0, 1)}^2 - \frac{1}{2\delta}. 
\end{align*} 
Choosing $\delta = m^2/\bond$, we obtain $2S_{\RR, m}(\zeta)\ge \frac{1}{2\bond}\|\zeta\|_{V_1^1(0, 1)}^2 - \frac{\bond}{2m^2}$. Together with $2D_{\RR, m}(\varphi)\ge \|\varphi\|_{V_1^1(\RR)}^2$, this shows that $E_{\RR, m}(\varphi, \zeta)$ controls $\|(\varphi, \zeta)\|_{\H_\RR}$ and so $E_{\RR, m}$ is weakly coercive on $\H_\RR$. Weak lower semicontinuity of $E_{\RR, m}$ on $\H_\RR$ follows from weak lower semicontinuity of the norms $\|\cdot\|_{V_1^1(\RR)}$ and $\|\cdot\|_{V_1^1(0, 1)}$ and the compact embedding $V_1^1(0, 1)\hookrightarrow L_1^2(0, 1)$. The existence of a minimizer to \eqref{eq:Slosh3D_VC2} now follows from the direct method of the calculus of variations. Finally, showing a minimizer and its $\mathrm{argmin}$ is a weak fundamental sloshing eigenpair of \eqref{eq:SloshRS} is similar to that of Theorem \ref{thm:Slosh2D_VC2} and we omit this proof for brevity. 
\end{proof} 

\begin{theorem}[Variational characterization, $m = 0$] \label{thm:Slosh3D_VC3} 
Let $\RR$ be a bounded Lipschitz domain in $\R^2$. There exists a weak fundamental sloshing eigenpair $(\omega_{0, 1}, \varphi_{0, 1}, \zeta_{0, 1})$ of \eqref{eq:SloshRS}, where $(\varphi_{0, 1}, \zeta_{0, 1})$ is a constrained minimizer of the following variational problem: 
\begin{equation*} 
\omega_{0, 1}\coloneqq \inf_{(\varphi, \zeta)\in\H_{\RR, 0}} \left\{E_{\RR, 0}(\varphi, \zeta)\colon G_\RR(\varphi, \zeta) = 1\right\}. 
\end{equation*} 
\end{theorem} 
\begin{proof} 
Note that $H_1^1(\RR)$ and $\os{H_1^1}(0, 1)$ are isomorphic to the space of all radially symmetric functions in $H^1(\D)$ and in $H_0^1(0, 1)$, respectively; see \cite[Section II.4]{Bernardi}. The proof is then similar to Theorem  \ref{thm:Slosh2D_VC3} and we omit the proof for brevity. 
\end{proof}


\subsection{Isoperimetric sloshing problem on shallow radially symmetric containers} \label{sec:Iso3D} 
Recall the class of admissible shape functions for shallow radially symmetric containers 
\[ \M_V = \left\{h\in \PC^1[0, 1]\colon \tr{$h\ge 0$ on $[0, 1]$; $\int_0^1 hr\, dr = V/2\pi$}\right\}. \] 
Similar to the case of shallow canals, we apply the shallow water theory and assume $\varphi(r, z) = \psi(r)$. Define the function spaces $\H\coloneqq V_1^1(0, 1)\times \os{V_1^1}(0, 1)$ and $\H_0\coloneqq\{(\psi, \zeta)\in H_1^1(0, 1)\times \os{H_1^1}(0, 1)\colon \int_0^1 \zeta r\, dr = 0\}$. We may approximate $\omega_{m, 1}$ as the infimum of the following one-dimensional constrained variational problem, thanks to Theorems \ref{thm:Slosh3D_VC2} and \ref{thm:Slosh3D_VC3}:  
\begin{align} \label{eq:Slosh3D_VCs} 
\omega_{m, 1}(h)\approx\Omega_{m, 1}(h)\coloneqq \inf_{(\psi, \zeta)\in\X_m}\left\{E_m(h; \psi, \zeta)\colon G_r(\psi, \zeta) = 1\right\}, 
\end{align} 
where $\X_0\coloneqq \H_0$, $\X_m\coloneqq \H$ for every $m = 1, 2, 3, \dots$, $G_r(\psi, \zeta)\coloneqq \int_0^1 \psi\zeta r\, dr$, and 
\begin{align*} 
E_m(h; \psi, \zeta) & \coloneqq \frac{1}{2}\int_0^1 h\bigg[(\psi')^2 + \frac{m^2}{r^2}\psi^2\bigg] r\, dr + \frac{1}{2}\int_0^1 \bigg[\zeta^2 + \frac{1}{\bond}\bigg((\zeta')^2 + \frac{m^2}{r^2}\zeta^2\bigg)\bigg]r\, dr. 
\end{align*} 

We then define the weak formulation of the one-dimensional pinned-edge shallow sloshing problem for every $m = 0, 1, 2, \dots$ on $[0, 1]$ as the weak form of the Euler-Lagrange equation of the constrained variational problem \eqref{eq:Slosh3D_VCs}. 

\begin{definition} 
Given $m = 0, 1, 2, \dots$, the weak shallow sloshing eigenpair $(\Omega_m, \psi_m, \zeta_m)$ $\in\R\times\X_m$ satisfies the following equation for all $(f, g)\in\X_m$: 
\begin{equation*}
\int_0^1 \bigg[h\bigg(\psi_m'f' + \frac{m^2}{r^2}\psi_mf\bigg) + \zeta g + \frac{1}{\bond}\bigg(\zeta'g' + \frac{m^2}{r^2}\zeta g\bigg)\bigg] r\, dr = \Omega_m\int_0^1 (\zeta_mf + \psi_mg) r\, dr.
\end{equation*} 
\end{definition} 

The corresponding one-dimensional boundary eigenvalue problem for every $m = 0, 1, 2, \dots$ is given by 
\begin{subequations} \label{eq:Slosh3Ds} 
\begin{alignat}{2} 
\label{eq:Slosh3Ds1} -\left(\frac{1}{r}(rh\psi_m')' - \frac{m^2}{r^2}h\psi_m\right) & = \Omega_m\zeta_m && \ \ \tr{ on $(0, 1)$}, \\ 
\label{eq:Slosh3Ds2} \zeta_m - \frac{1}{\bond}\left(\frac{1}{r}(r\zeta_m')' - \frac{m^2}{r^2}\zeta_m\right) & = \Omega_m\psi_m && \ \ \tr{ on $(0, 1)$}, \\ 
\label{eq:Slosh3Ds3} \left(h\psi_m'\right)(1) & = \zeta_m(1) = 0. 
\end{alignat} 
\end{subequations} 
For $m = 0$, we must impose the necessary condition $\int_0^1 \zeta_0r\, dr = 0$ as well. 

We now state the \emph{isoperimetric sloshing problem for shallow radially symmetric containers}: For every $m = 0, 1, 2, \dots$, find the shape function $h\in\M_V$ that maximizes $\Omega_{m, 1}$, \ie solve
\begin{equation} \label{eq:Iso3D} 
\sup_{h\in\M_V} \Omega_{m, 1}(h). 
\end{equation} 
Similarly to the case of shallow canals (see Theorem \ref{thm:Iso2D_opt}), we obtain the following sufficient condition for solving \eqref{eq:Iso3D}. 

\begin{theorem} \label{thm:Iso3D_opt} 
Given $m = 0, 1, 2, \dots$, if there exists $h_m^*\in\M_V$ such that $(\Omega_{m, 1}(h_m^*), \psi_m^*, \zeta_m^*)\in\R\times\X_m$ is a fundamental weak shallow sloshing eigenpair with $\psi_m^*$ satisfying 
\begin{equation} \label{eq:Iso3D_opt} 
\left(\left(\psi_m^*\right)'\right)^2 + \frac{m^2}{r^2}(\psi_m^*)^2 = c_m^2 \ \ \tr{ on $r\in [0, 1)$} 
\end{equation} 
for some constant $c_m^2\neq 0$, then $h_m^*$ is a maximizer of $\Omega_{m, 1}$. 
\end{theorem} 

\begin{remark} 
Formally, \eqref{eq:Iso3D_opt} can be interpreted as the first-order optimality condition for \eqref{eq:Iso3D}. The formal computation is similar to that in Remark \ref{rm:Iso2D_opt} and we omit this for brevity. 
\end{remark}


\subsection{Finite surface tension: proof of Theorems \ref{thm:Iso3D_1} and \ref{thm:Iso3D_0}} \label{sec:Iso3D_OC} 
We now solve \eqref{eq:Iso3D} for $m = 1$ and $m = 0$ only. Throughout this subsection, we denote by $h$ the maximizing cross-section, $(\Omega, \psi, \zeta)$ its corresponding fundamental shallow sloshing eigenpair, and $\lambda\coloneqq \Omega^2$ for notational convenience.  

\begin{proof}[\tb{Proof of Theorem \ref{thm:Iso3D_1}}] 
Set $m = 1$. In this case, we must have $\psi(0) = 0 = \zeta(0)$ since functions in $V_1^1(0, 1)$ have null trace at $r = 0$ \cite{mercier:1982}. Choosing $c_1^2 = 2$ in \eqref{eq:Iso3D_opt}, one such solution satisfying $\psi(0) = 0$ is $\psi(r) = r$. We first solve for $\zeta$. Substituting $\psi(r) = r$ into \eqref{eq:Slosh3Ds2} for $m = 1$ and rearranging, we find $\zeta$ satisfies the following nonhomogeneous Bessel-type boundary value problem: 
\begin{equation} \label{eq:Iso3D_11} 
r^2\zeta'' + r\zeta' - (\bond\, r^2 + 1)\zeta = -\Omega \bond\, r^3; \ \ \zeta(0) = \zeta(1) = 0. 
\end{equation} 
Define $\kappa\coloneqq \sqrt{\bond}$ and introduce a scaled coordinate $s\coloneqq \kappa r$. Letting $y(s) = \zeta(r) = \zeta(s/\kappa)$, we see that \eqref{eq:Iso3D_11} transforms to 
\begin{equation*} 
s^2y'' + sy' - \left(s^2 + 1\right)y = -\frac{\Omega}{\kappa}s^3; \ \ y(0) = y(\kappa) = 0. 
\end{equation*} 
Using the method of undetermined coefficients, the solution is given by 
\[ y(s) = \Omega\left[\frac{s}{\kappa} - \frac{I_1(s)}{I_1(\kappa)}\right]\ \ \tr{ or } \ \ \zeta(r) = \Omega\left[r - \frac{I_1(\kappa r)}{I_1(\kappa)}\right], \] 
where $I_\nu$ is the modified Bessel function of the first kind of order $\nu$. We discard the modified Bessel function of the second kind $K_1$ since we require $\zeta(0) = 0$. 

Next we solve for $h$. Substituting $\psi$ and $\zeta$ into \eqref{eq:Slosh3Ds1} for $m = 1$ and \eqref{eq:Slosh3Ds3} and rearranging suitably, we get 
\[ h' = -\lambda\left[r - \frac{I_1(\kappa r)}{I_1(\kappa)}\right]; \ \ h(1) = 0. \] 
Using the integration formula \cite[Eq.~10.43.1]{NIST} with $\nu = 1$, the solution is given by
\[ h(r) = \frac{\lambda}{2}(1 - r^2) - \frac{\lambda}{\kappa\, I_1(\kappa)}\Big[I_0(\kappa) - I_0(\kappa r)\Big]. \]  
Imposing the volume constraint $\int_0^1 hr\, dr = V/2\pi$ yields 
\begin{align*} 
\frac{V}{2\pi} & = \frac{\lambda}{8} - \frac{\lambda}{\kappa\, I_1(\kappa)}\left[\frac{I_0(\kappa)}{2} - \frac{I_1(\kappa)}{\kappa}\right] = \frac{\lambda}{8} - \frac{\lambda I_2(\kappa)}{2\kappa\, I_1(\kappa)}, 
\end{align*} 
where we use the integration formula \cite[Eq.~10.43.1]{NIST} with $\nu = 0$ and the recurrence relation \cite[Eq.~10.29.1]{NIST} with $\nu = 1$. Rearranging for $\lambda$ gives $\lambda = \lambda_{1, 1}^*$ as defined in \eqref{eq:Iso3D_1_eigs}. Finally, it can be shown that $h\in\M_V$ as $h(1) = 0$ and $h' < 0$ on $(0, 1)$; the latter follows from the fact that $I_1(z)/z$ is strictly increasing for $z > 0$ which is an immediate consequence of the derivative formula $[I_1(z)/z]' = I_2(z)/z$ \cite[Eq.~10.29.4]{NIST}. 
\end{proof}

\begin{proof}[\tb{Proof of Theorem \ref{thm:Iso3D_0}}] 
Set $m = 0$. Choosing $c_0^2 = 1$ in \eqref{eq:Iso3D_opt}, one possible solution is $\psi = r - d_0$ for some $d_0\in\R$. We first solve for $\zeta$. Substituting $\psi(r) = r - d_0$ into \eqref{eq:Slosh3Ds2} for $m = 0$ and rearranging, we find $\zeta$ satisfies the following nonhomogeneous Bessel-type equation: 
\begin{equation} \label{eq:Iso3D_01} 
r^2\zeta'' + r\zeta' - \bond\, r^2\zeta = -\Omega\bond\, r^2(r - d_0). 
\end{equation} 
Define $\kappa\coloneqq \sqrt{\bond}$ and introduce a scale coordinate $s\coloneqq \kappa r$. Letting $y(s) = \zeta(r) = \zeta(s/\kappa)$, we see that \eqref{eq:Iso3D_01} transforms to 
\begin{equation} \label{eq:Iso3D_02} 
s^2y'' + sy' - s^2y = -\frac{\Omega}{\kappa} s^2(s - d_0\kappa). 
\end{equation} 
The associated homogeneous equation has solutions $I_0(s)$ and $K_0(s)$. Motivated by the particular solution of $\zeta$ from the proof of Theorem \ref{thm:Iso3D_1}, we guess a particular solution of the form $y_p(s) = Bs + C + F(s)$ for some constants $B, C$ and function $F(s)$. Substituting $y_p$ into \eqref{eq:Iso3D_02}, we get $B = \Omega/\kappa$, $C = -\Omega d_0$, and $F(s)$ satisfies the following Struve-type equation: 
\begin{equation} \label{eq:Iso3D_03} 
s^2F'' + sF' - s^2F = -Bs = -\frac{\Omega s}{\kappa}. 
\end{equation} 
Comparing \eqref{eq:Iso3D_03} with \cite[Eqs.~11.2.9, 11.2.10]{NIST} with $\nu = 0$, we may take $F(s) = -\Omega\pi\Lb_0(s)/(2\kappa)$, where $\Lb_0$ is the modified Struve function of the first kind. Thus, the general solution of \eqref{eq:Iso3D_02} is given by
\[ y(s) = c_1I_0(s) + \frac{\Omega}{\kappa}\left[s - d_0\kappa - \frac{\pi}{2}\Lb_0(s)\right], \ \ c_1\in\R, \] 
where we discard $K_0$ since we require $\zeta(0) = y(0) < \infty$. Imposing the boundary condition $\zeta(1) = y(\kappa) = 0$, converting from $s$ to $r$, and rearranging suitably, we obtain 
\begin{align*} 
\zeta(r) & = \Omega\left[r - \frac{I_0(\kappa r)}{I_0(\kappa)} - d_0\left(1 - \frac{I_0(\kappa r)}{I_0(\kappa)}\right) - \frac{\pi}{2\kappa}\left(\Lb_0(\kappa r) - \frac{\Lb_0(\kappa)}{I_0(\kappa)}I_0(\kappa r)\right)\right] \\ 
& \eqqcolon \Omega\left[\zeta_1(r) - d_0\zeta(r) - \frac{\pi}{2\kappa}\zeta_3(r)\right]. 
\end{align*} 

Next we find $d_0$ by imposing the necessary condition $\int_0^1 \zeta r\, dr = 0$. Using the integral formula \cite[Eq.~10.43.1 and 11.7.3]{NIST} with $\nu = 0$, we obtain antiderivatives of $\zeta_1r$, $\zeta_2r$, and $\zeta_3r$: 
\begin{equation} \label{eq:Iso3D_04} 
\begin{alignedat}{1} 
A_1(r) & \coloneqq \int \zeta_1r\, dr = \frac{r^3}{3} - \frac{I_1(\kappa r)r}{\kappa\, I_0(\kappa)}, \ \ A_2(r) \coloneqq \int \zeta_2r\, dr = \frac{r^2}{2} - \frac{I_1(\kappa r)r}{\kappa\, I_0(\kappa)} \\ 
A_3(r) & \coloneqq \int \zeta_3r\, dr = \frac{\Lb_1(\kappa r)r}{\kappa} - \frac{\Lb_0(\kappa)I_1(\kappa r)r}{\kappa\, I_0(\kappa)}. 
\end{alignedat} 
\end{equation} 
It is clear that $A_1(0) = A_2(0) = A_3(0) = 0$. Also, $A_3(1) = -\Upsilon(\kappa)/\kappa I_0(\kappa)$ with $\Upsilon(\kappa)$ defined in \eqref{eq:Iso3D_0_d0} and the recurrence relation \cite[Eq.~10.29.1]{NIST} with $\nu = 1$ gives $A_2(1) = I_2(\kappa)/2I_0(\kappa)$. Thus, we obtain 
\begin{equation} \label{eq:Iso3D_05} 
\int_0^1 \zeta r\, dr = \Omega c_0\left[A_1(1) - d_0\, A_2(1) - \frac{\pi}{2\kappa}\, A_3(1)\right] = 0. 
\end{equation} 
Rearranging for $d_0$ yields the desired expression for $d_0$ as defined in \eqref{eq:Iso3D_0_d0}. 

Next we solve for $h$. Substituting $\psi$ and $\zeta$ into \eqref{eq:Slosh3Ds1} for $m = 0$ and \eqref{eq:Slosh3Ds3}, we get 
\[ (rh)' = -\lambda\left[\zeta_1r - d_0\zeta_2r - \frac{\pi}{2\kappa}\zeta_3r\right]; \ \ h(1) = 0. \] 
Integrating once using \eqref{eq:Iso3D_04}, imposing $h(1) = 0$ together with \eqref{eq:Iso3D_05}, and grouping terms suitably, the solution is given by 
\[ h(r) = \frac{\lambda}{3}\left[\frac{3d_0r}{2} - r^2\right] + \frac{\lambda (1 - d_0)I_1(\kappa r)}{\kappa\, I_0(\kappa)} + \frac{\lambda \pi}{2\kappa^2}\left[\Lb_1(\kappa r) - \frac{\Lb_0(\kappa)I_1(\kappa r)}{I_0(\kappa)}\right]. \] 
Integrating $I_1(\kappa r)$ and $\Lb_1(\kappa r)$ using the integral formulas \cite[Eq.~10.4.3]{NIST} with $\nu = 1$ and \cite[Formula 1.3]{gaunt:2018} with $\nu = 1$, respectively, the volume constraint $\int_0^1 hr\, dr = V/2\pi$ becomes 
\begin{align*} 
\frac{V}{2\pi} & = \frac{\lambda}{3}\left[\frac{d_0}{2} - \frac{1}{4}\right] + \frac{\lambda(1 - d_0)\pi\Upsilon(\kappa)}{2\kappa^2I_0(\kappa)} \\ 
& \qquad + \frac{\lambda\pi}{2\kappa^2}\left[\frac{\kappa^2}{6\pi}\, {}_2F_3\left(1, 2; \frac{3}{2}, \frac{5}{2}, 3; \frac{\kappa^2}{4}\right) - \frac{\pi\Lb_0(\kappa)\Upsilon(\kappa)}{2\kappa\, I_0(\kappa)}\right]. 
\end{align*} 
Factoring $1/36$ from the right side of the equation above, rearranging for $\lambda$, and simplifying, we obtain $\lambda = \lambda_{0, 1}^*$ as defined in \eqref{eq:Iso3D_0_eigs}. Finally, we observe that $h\ge 0$ on $[0, 1]$ and so $h\in\M_V$. 
\end{proof}

\begin{figure}[t] 
\begin{center}
	\includegraphics[width = 0.49\textwidth]{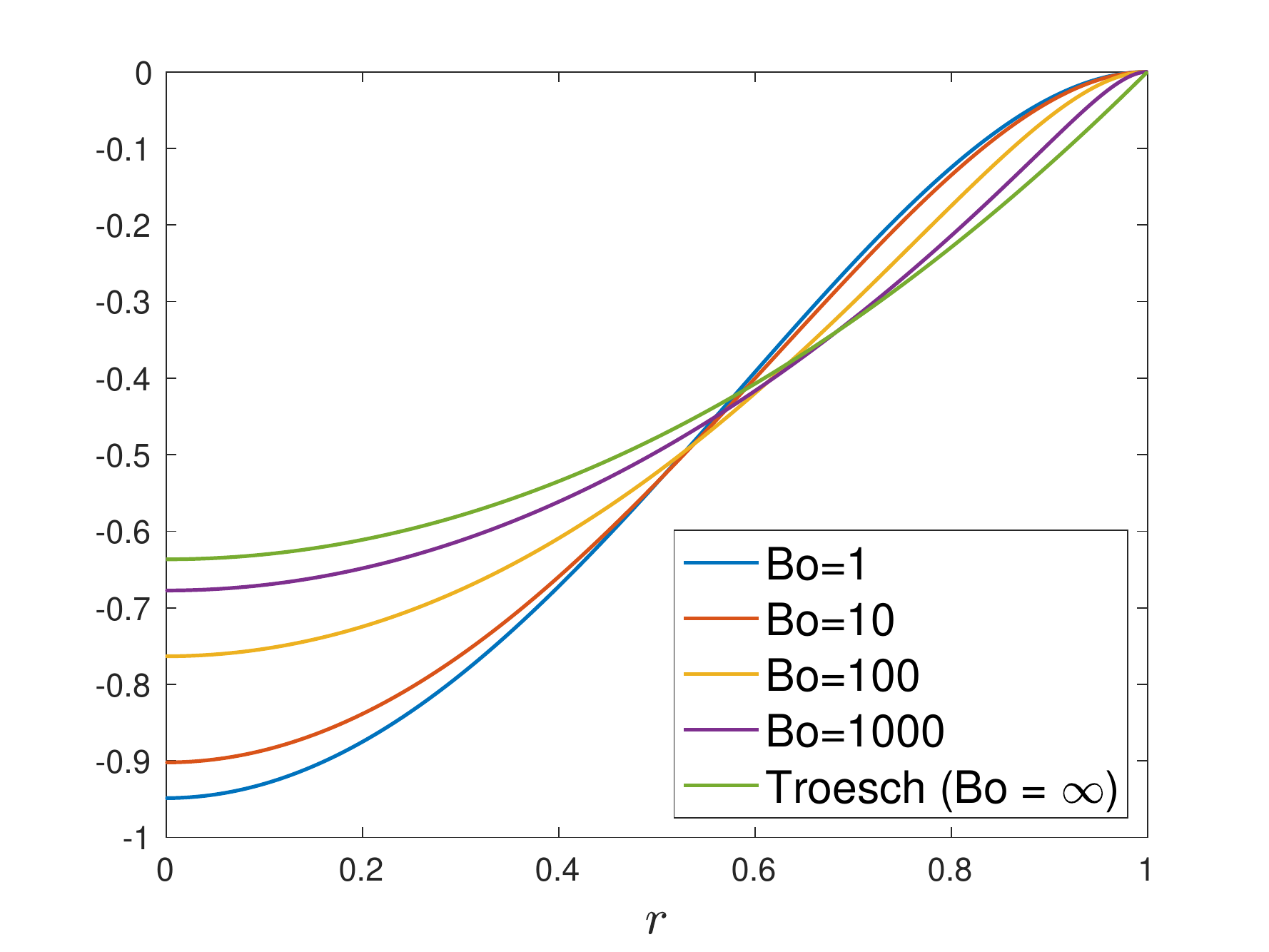} 
	\includegraphics[width = 0.49\textwidth]{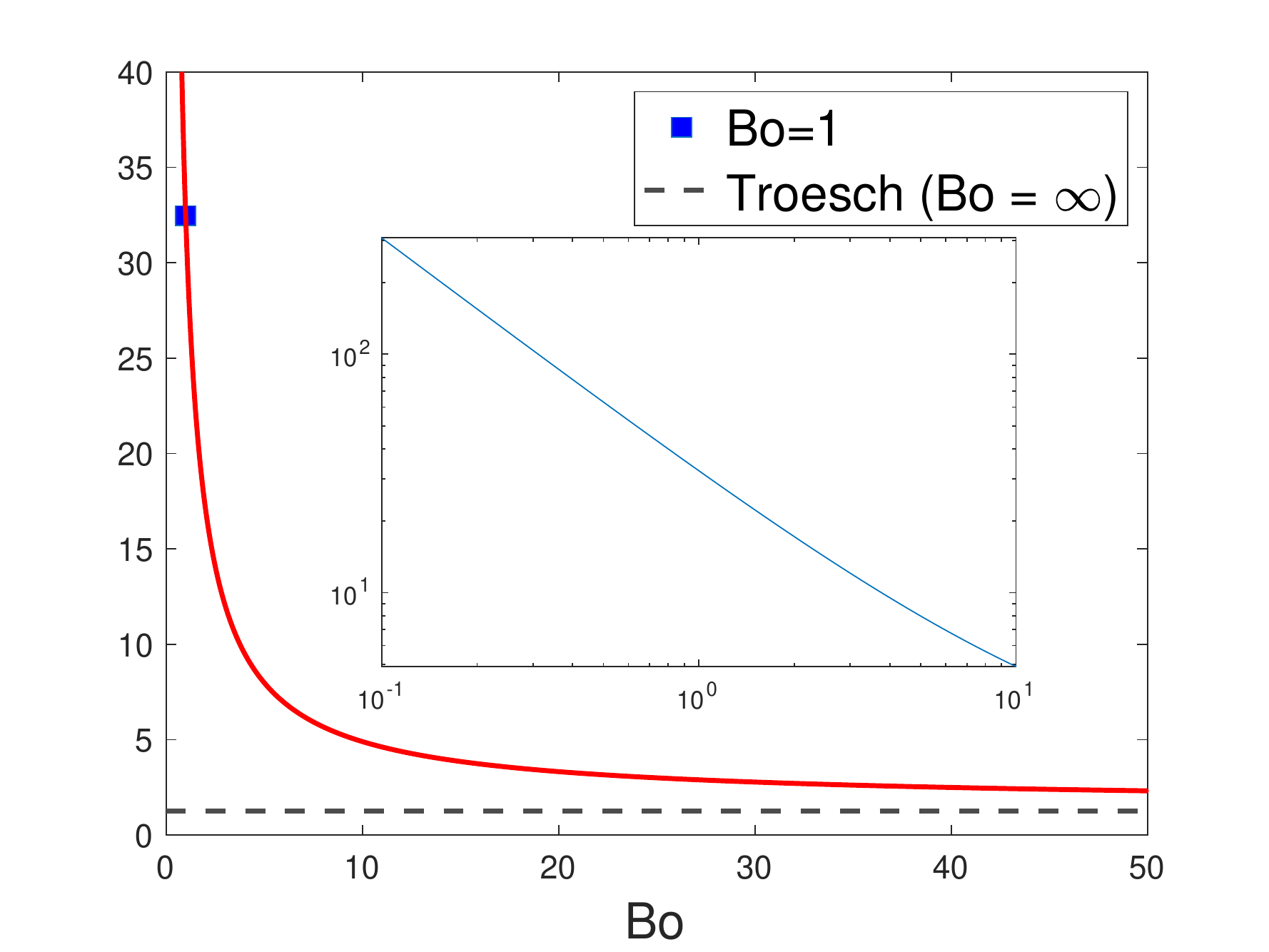} 
	\includegraphics[width = 0.49\textwidth]{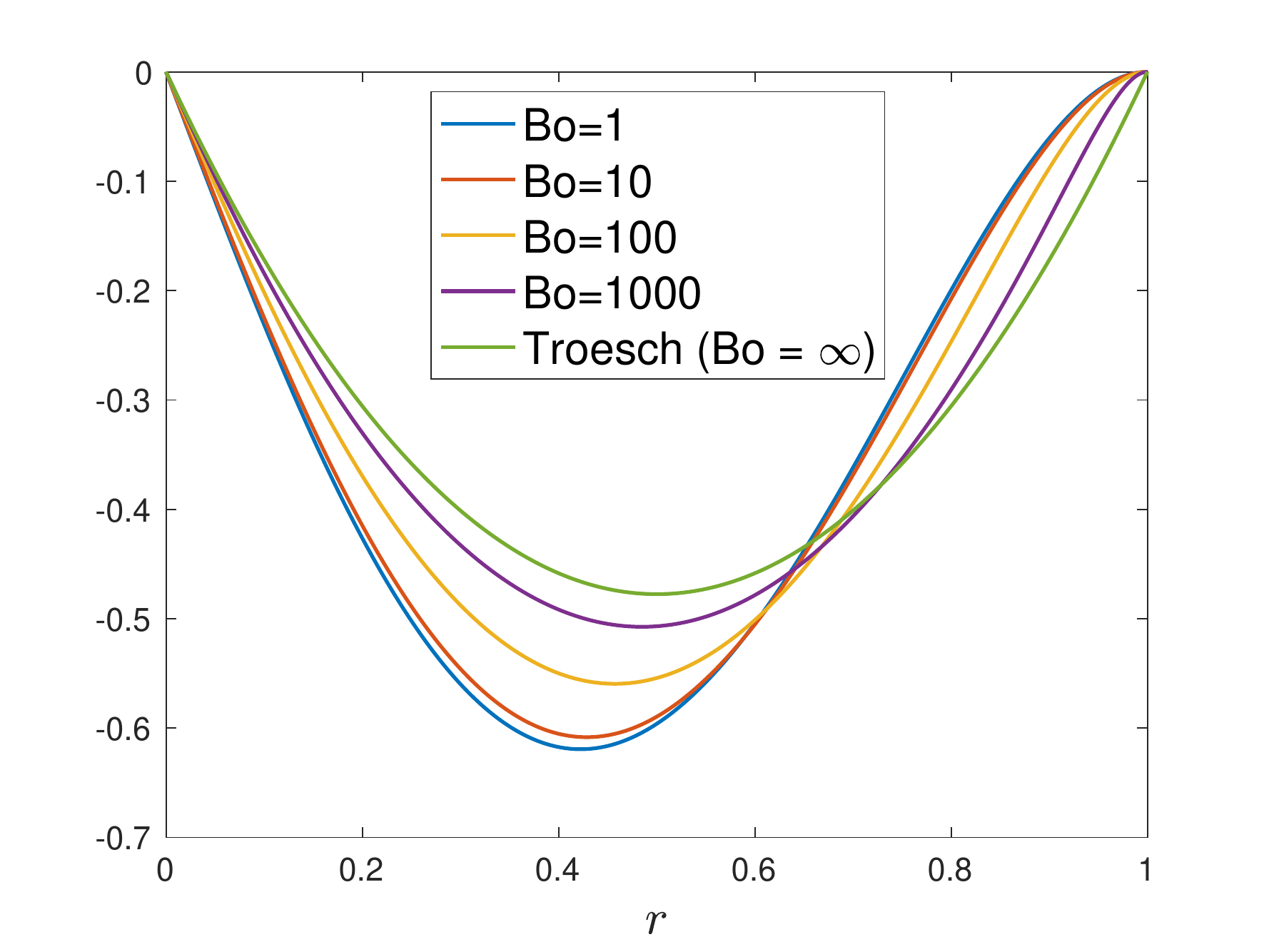} 
	\includegraphics[width = 0.49\textwidth]{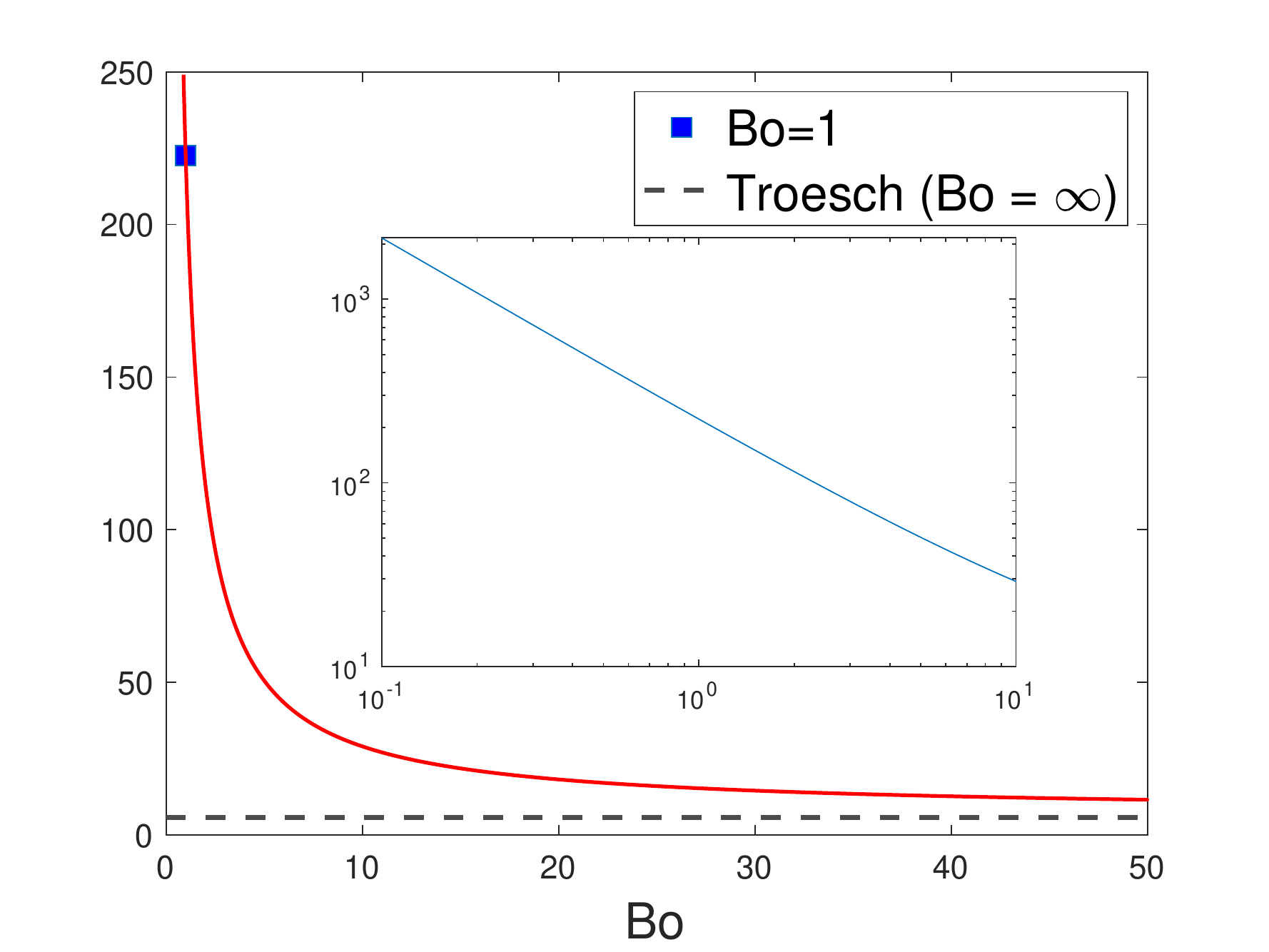}
\caption{An illustration of the zero surface tension limit for $m = 1$ \tb{(top)} and $m = 0$ \tb{(bottom)}, both with $V = 1$. \tb{(Left)} The maximizing cross-section plotted for varying $\bond$ (see \eqref{eq:Iso3D_1_shape} and \eqref{eq:Iso3D_0_shape}) and for $\bond = \infty$ (zero surface tension). \tb{(Right)} The squared maximal sloshing frequency \eqref{eq:Iso3D_1_eigs} and \eqref{eq:Iso3D_0_eigs} plotted as a function of $\bond$, with $\lambda_{1, 1, \bond = \infty} = 4/\pi$ and $\lambda_{0, 1, \bond = \infty} = 18/\pi$ (dashed lines). The inset is a log-log plot of the squared maximal sloshing frequency for $\bond\in [0.1, 10]$.} 
\label{fig:ST_m_1} 
\end{center}
\end{figure} 


\subsection{Zero surface tension limit ($\bond\to\infty$)} \label{sec:Iso3D_zeroST} 
In this subsection, we show that Troesch's optimal containers without surface tension for $m = 1$ and $m = 0$ are the zero surface tension limit of the optimal containers from Theorems \ref{thm:Iso3D_1} and \ref{thm:Iso3D_0}, respectively. Moreover, the map $\bond\mapsto\lambda_{m, 1}^*(\bond)$ is strictly decreasing on $(0, \infty)$. Figure \ref{fig:ST_m_1} illustrates these results with $m = 1$ (top) and $m = 0$ (bottom). For $\bond = 1$, we get $\lambda_{1, 1}^*/\lambda_{1, 1, \bond = \infty}^*$ and $\lambda_{0, 1}^*/\lambda_{0, 1, \bond = \infty}^*$ to be approximately 25.5 and 38.8, \ie \emph{the squared maximal sloshing frequency increases drastically when capillary and gravitational forces are comparable}. In Figure \ref{fig:ST_m_1}(right), the log-log plots reveal that $\lambda_{1, 1}^*(\bond)\propto\bond^{-0.86}$ and $\lambda_{0, 1}^*(\bond)\propto\bond^{-0.911}$ for $\bond\in [0.1, 10]$. 

To prove the claimed zero surface tension limit, we will repeatedly use the asymptotic behavior of $I_\nu$ \cite[Eq.~10.30.4]{NIST}: 
\begin{equation} \label{eq:asymp} 
I_\nu(\kappa)\sim \frac{e^\kappa}{\sqrt{2\pi\kappa}}, \ \ \kappa\to\infty, \, \nu\in\R. 
\end{equation} 

\begin{corollary}[$m = 1, \bond\to\infty$] \label{thm:Iso3D_zeroST_1}
Let $\lambda_{1, 1}^*$ and $h_1^*$ be defined as in Theorem \ref{thm:Iso3D_1}. The map $\bond\mapsto\lambda_{1, 1}^*(\bond)$ is strictly decreasing on $(0, \infty)$ and $\lambda_{1, 1}^*(\bond)\to 4V/\pi$ as $\bond\to\infty$. Moreover, $h_1^*(\bond; r)\to 2V(1 - r^2)/\pi$ as $\bond\to\infty$ for every $r\in [0, 1]$. 
\end{corollary} 
\begin{proof} 
Define $\kappa\coloneqq \sqrt{\bond} > 0$. From \eqref{eq:Iso3D_1_eigs}, we write $\lambda_{1, 1}^*$ as  
\[ \lambda_{1, 1}^*(\bond) = \lambda_{1, 1}^*(\kappa^2) = \frac{4V}{\pi}\Big(1 - 4Y(\kappa)\Big)^{-1}, \tr{ with $Y(\kappa) = \frac{I_2(\kappa)}{\kappa I_1(\kappa)}$}. \] 
Proving the first statement of the corollary is equivalent to proving that $Y$ is strictly decreasing on $(0, \infty)$, with range $(0, 1/4)$. The strict monotonicity and limit as $\kappa\to 0^+$ were both proved in \cite{simpson:1984}, and the limit at infinity follows from \eqref{eq:asymp}. To prove the second statement, we need only show the second term in \eqref{eq:Iso3D_1_shape} vanishes in the limit of $\bond\to\infty$ for any $r\in [0, 1]$, but this follows easily from \eqref{eq:asymp}. 
\end{proof}   

For the next corollary, we observe that the map $\bond\mapsto\lambda_{0, 1}^*(\bond)$ is strictly decreasing on $(0, \infty)$ as well. 

\begin{corollary}[$m = 0, \bond\to\infty$] \label{thm:Iso3D_zeroST_0}
Let $\lambda_{0, 1}^*$ and $h_0^*$ be defined as in Theorem \ref{thm:Iso3D_0}. Then $\lambda_{0, 1}^*(\bond)\to 18V/\pi$ as $\bond\to\infty$. Moreover, $h_0^*(\bond; r)\to 6(r - r^2)/\pi$ as $\bond\to\infty$ for every $r\in [0, 1]$. 
\end{corollary} 
\begin{proof} 
Define $\kappa\coloneqq\sqrt{\bond} > 0$. From \eqref{eq:Iso3D_0_eigs}, we write $\lambda_{0, 1}^*$ as 
\[ \lambda_{0, 1}^*(\bond) = \lambda_{0, 1}^*(\kappa^2) = \frac{18V}{\pi}\left[Y_1(\kappa) + 18(1 - d_0)Y_2(\kappa) + \frac{Y_3(\kappa)}{36}\right]^{-1}, \] 
where $Y_1(\kappa)\coloneqq 6d_0 - 3$, $Y_2(\kappa)\coloneqq \pi\Upsilon(\kappa)/(\kappa^2I_0(\kappa))$, and 
\begin{equation*} 
Y_3(\kappa)\coloneqq \frac{I_0(\kappa)\int_0^\kappa s\Lb_1(s)\, ds - \Lb_0(\kappa)\int_0^\kappa sI_1(s)\, ds}{\kappa^4I_0(\kappa)}. 
\end{equation*} 
The integral expression for $Y_3(\kappa)$ follows from noticing the last two terms in $\lambda_{0, 1}^*$ are obtained as the definite integral of the last term in $h_0^*$; see the proof of Theorem \ref{thm:Iso3D_0}. From \cite[Corollary 2.6]{gaunt:2014} with $\nu = 1$ and \eqref{eq:Iso3D_0_d0}, we find 
\begin{equation*}
\frac{2}{\kappa^2} < \frac{\pi\Upsilon(\kappa)}{\kappa^2I_2(\kappa)} < \frac{8}{3\kappa^2} \ \ \tr{ for all $\kappa > 0$} \implies \lim_{\kappa\to\infty} \frac{\pi\Upsilon(\kappa)}{\kappa^2I_2(\kappa)} = 0, 
\end{equation*} 
and together with \eqref{eq:asymp}, we obtain 
\begin{align*} 
\lim_{\kappa\to\infty} d_0 = \lim_{\kappa\to\infty} \left[\frac{2I_0(\kappa)}{3I_2(\kappa)} - \frac{2I_1(\kappa)}{\kappa\, I_2(\kappa)} + \frac{\pi\Upsilon(\kappa)}{\kappa^2I_2(\kappa)}\right] = \frac{2}{3}. 
\end{align*} 
The two limits above yield $Y_1(\kappa)\to 1$ and $Y_2(\kappa)\to 0$ as $\kappa\to\infty$. To compute the limit of $Y_3$, let $\Mb_\nu\coloneqq \Lb_\nu - I_\nu$ be the modified Struve function of the second kind of order $\nu$ \cite[Eq.~11.2.6]{NIST}. Integrating by parts using the derivative formulas $\Lb_0'(r) = \Lb_1(r) + 2/\pi$ and $I_0'(r) = I_1(r)$ \cite[Eqs.~10.29.3, 11.4.33]{NIST}, we obtain
\[ Y_3(\kappa) = \frac{\Mb_0(\kappa)}{\kappa^4}\frac{\int_0^\kappa I_0(s)\, ds}{I_0(\kappa)} - \frac{1}{\kappa^4}\int_0^\kappa \Mb_0(s)\, ds - \frac{1}{\pi\kappa^2}. \] 
Combining \eqref{eq:asymp} and a standard asymptotic analysis using integration by parts yield \[ \int_0^\kappa I_0(s)\, ds\sim \frac{e^\kappa}{\sqrt{2\pi\kappa}} \ \ \tr{ as $\kappa\to\infty$}. \] 
This together with \eqref{eq:asymp} and the limiting forms \cite[Eqs.~11.6.1 (with $\nu = 0$), 11.6.4]{NIST} show $Y_3(\kappa)\to 0$ as $\kappa\to\infty$. The first statement of the corollary now follows. 

It remains to establish the zero surface tension limit of $h_0^*(\bond; r)$. Comparing \eqref{eq:Iso3D_0_shape} with the desired expression, we need only show the last two terms in \eqref{eq:Iso3D_0_shape} vanish in the limit of $\bond\to\infty$ for any $r\in [0, 1]$. This is evident for $r = 0$ and $r = 1$. For $r\in (0, 1)$, this follows from \eqref{eq:asymp} for the second term and from Mathematica for the third term. 
\end{proof}


\section{Discussion} \label{sec:disc} 
Assuming a flat equilibrium free surface and a pinned contact line, we considered the problem of maximizing the fundamental sloshing frequency over two classes of shallow containers: canals with a given free surface width and cross-sectional area, and radially symmetric containers with a given rim radius and volume. In addition to including the effects of surface tension, for canals, we extended the problem of two-dimensional sloshing in the vertical plane to traveling sinusoidal waves along the canal, which introduced the wavenumber $\alpha\ge 0$ as an additional parameter. 

In subsections \ref{sec:canal_variational} and \ref{sec:RS_variational}, we established a new variational characterization of fluid sloshing with surface tension for a pinned contact line. Combining this result with the shallow water theory, we approximated the fundamental sloshing frequency for shallow containers as the infimum of a one-dimensional constrained variational problem; see \eqref{eq:Slosh2D_VCs} and \eqref{eq:Slosh3D_VCs}. We defined the pinned-edge linear shallow sloshing problem as the corresponding Euler-Lagrange equations; see \eqref{eq:SloshCanalS} and \eqref{eq:Slosh3Ds}. Based on a simple observation of the specific form of the energy functional, we derived a sufficient condition and outlined a strategy for solving the isoperimetric sloshing problem; see Theorems \ref{thm:Iso2D_opt} and \ref{thm:Iso3D_opt}. 

In the absence of surface tension, to our surprise, we found that the optimal shallow canal for every $\alpha > 0$ is rectangular; see Theorem \ref{thm:Iso2D_zero}. In the presence of surface tension, we found explicit solutions for both the optimal shallow containers and the corresponding maximal sloshing frequency; see Theorems \ref{thm:Iso2D_0}-\ref{thm:Iso3D_0}. Interestingly, the optimal shallow canal for any $\alpha \ge 0$ is symmetric, but it is convex only for the case $\alpha = 0$. On the other hand, the optimal shallow radially symmetric container is not convex in both cases $m = 1$ and $m = 0$. For each of these optimal shallow containers, we found that the corresponding squared maximal sloshing frequency is a decreasing function of $\bond$ and including the effects of surface tension gives a significantly larger squared maximal sloshing frequency. Finally, because all our results are explicit, we proved that the limit of the solution (both the maximizing cross-section and its squared maximal sloshing frequency) to the isoperimetric sloshing problem with surface tension, as surface tension vanishes, \ie $\bond\to\infty$, is the solution to the isoperimetric sloshing problem without surface tension. 

We made three crucial assumptions that allowed us to solve the isoperimetric sloshing problem explicitly: (1) the equilibrium free surface is flat, (2) the contact line is pinned, and (3) the container is shallow. It would be interesting to extend our results by removing one or more of these assumptions, such as considering a curved equilibrium free surface (meniscus), other dynamic contact line boundary conditions such as Hocking's wetting boundary condition \cite{hocking:1987aa,hocking:1987bb}, or more general three-dimensional containers. 

Two additional avenues for future work are finding lower bounds for natural sloshing frequencies and investigating other geometrical constraints for the container shape. Troesch considered the isoperimetric sloshing problem of finding the container shape that minimizes the first and second sloshing frequencies among shallow convex containers \cite{troesch:1967isotrap}. Troesch proved that these optimal containers are trapezoidal containers \cite{troesch:1967integral} and obtained the following result: For planar sloshing in symmetric canals, the optimal container is rectangular for $\lambda_{0, 1}$ and triangular for $\lambda_{0, 2}$. For radially symmetric containers, the optimal container is cylindrical for $\lambda_{1, 1}$ and conical for $\lambda_{0, 1}$. Kuzanek studied a similar isoperimetric sloshing problem of maximizing the natural sloshing frequencies $\lambda_n$ on symmetric shallow canals, where he replaced the area constraint with an arc length constraint \cite{kuzanek:1974iso,kuzanek:1974existence}. Kuzanek established the existence of a unique optimal container for each $\lambda_n$, proved that they are convex, obtained the optimal containers numerically, and conjectured that they do not have vertical side walls. We would be interested in determining if these results continue to hold in the presence of surface tension. 

In the absence of surface tension, several results about the location of high spots, \ie the maximal elevation of the free surface height $\xi$, were obtained in \cite{kulczycki:2009,kulczycki:2011,kulczycki:2012} for the fundamental sloshing mode. For a planar domain whose wetted boundary $\B$ is the graph of a negative $C^2$ function on $\F$ and $\B$ form nonzero angles with $\F$ at their common endpoints, the high spot is located on $\del\F$. A similar result holds for finite canals whose vertical cross-sections satisfy the same condition. For the ice-fishing problem, \ie sloshing in the lower half-plane with $\F = \{(x,0)\colon |x| < a\}$ and $\F = \{(x, y, 0)\colon x^2 + y^2 < a^2\}$ for the two- and three-dimensional case respectively and $a > 0$, the high spot is located in the interior of $\F$. For a radially symmetric, convex, bounded container $\D\subset\F\times (-\infty, 0)$, the high spot is located on $\del\F$. All these results rely on the property that the free surface height is proportional to the trace of the fundamental sloshing mode $\Phi_1$ on $\bar\F$ if the fluid oscillates freely with the fundamental sloshing frequency, which is no longer true in the presence of surface tension due to the curvature term $\Delta_\F\xi$ in \eqref{eq:SloshST4}. In recent joint work with Nathan Willis, we used computational methods to study high spots for the ice-fishing problem with surface tension \cite{willis:2022}. It would be interesting to investigate the high spot problem with surface tension on bounded containers. 

\subsubsection*{Acknowledgements.} We would like to thank Emma Coates, Emily Dryden, Calvin Khor, Robert Viator, and Nathan Willis for stimulating discussions. We are also grateful to the referees for their valuable comments and suggestions which greatly improved the manuscript. 

\printbibliography 

\end{document}